\documentclass{amsart}

\input{arxiv.input}

\usepackage{graphicx}
\usepackage{amssymb,  amsmath}
\usepackage[vcentermath]{youngtab}
\usepackage{comment}
\usepackage{epstopdf,hyperref}

\newcommand{\refer}[1]{{\tag{See \eqref{#1}, p.~\pageref{#1}}}}

\usepackage{color}
\definecolor{lightblue}{rgb}{0.6,0.6,0.9}
\definecolor{vlightblue}{rgb}{0.80,0.80,1}
\definecolor{vlightred}{rgb}{1,0.95,0.95}

\newcommand{\note}[1]{}

\newcommand{\mysize}[1]{\text{{\small #1}}}

\newcommand{\backblue}{\colorbox{vlightblue}{\mysize{\phantom{2}}}}
\newcommand{\backred}{\colorbox{vlightred}{\mysize{\phantom{2}}}}

\newcommand{\complementary}[2]{S_{#1}{(#2)}}

\newcommand{\One}{\bar{1}}
\newcommand{\Zero}{\bar{0}}

\newcommand{\shift}{\delta}

\newcommand{\calli}[1]{{\mathcal{#1}}}
\newcommand{\AAA}{\mathcal{A}}
\newcommand{\BBB}{\mathcal{B}}
\newcommand{\DDD}{\mathcal{D}}
\newcommand{\QQQ}{\mathcal{Q}}
\newcommand{\XXX}{\mathcal{X}}
\newcommand{\YYY}{\mathcal{Y}}

\renewcommand{\neg}[1]{{\underline{#1}}}

\newcommand{\negBBB}{{\neg{\BBB}}}

\newcommand{\calplus}{\oplus}
\newcommand{\caltimes}{\otimes}

\newcommand{\calxxp}[1]{{X^{#1}}}

\newcommand{\calexp}[1]{{e^{#1}}}

\newcommand{\ud}[1]{{\text{d}{#1}}}
\newcommand{\primezeta}{\overset{\circ}{\zeta}}
\newcommand{\floor}[1]{{\lfloor #1 \rfloor}}

\DeclareMathOperator{\sortO}{sort}
\newcommand{\sort}[1]{\sortO\!{\left({#1}\right)}}
\newcommand{\symp}[2]{\mathfrak{p}_{#1}{\left(#2\right)}}
\newcommand{\symh}[2]{\mathfrak{h}_{#1}{\left(#2\right)}}
\newcommand{\syms}[2]{\mathfrak{s}_{#1}{\left(#2\right)}}
\newcommand{\symm}[2]{\mathfrak{m}_{#1}{\left(#2\right)}}
\newcommand{\syme}[2]{\mathfrak{e}_{#1}{\left(#2\right)}}
\newcommand{\sympsymb}[1]{\mathfrak{p}_{#1}}
\newcommand{\symhsymb}[1]{\mathfrak{h}_{#1}}
\newcommand{\symssymb}[1]{\mathfrak{s}_{#1}}
\newcommand{\symmsymb}[1]{\mathfrak{m}_{#1}}

\newcommand{\multipart}[2]{\langle{#1}^{#2}\rangle}
\newcommand{\EulerGamma}{\gamma}
\newcommand{\StieltjesGamma}[1]{\gamma_{#1}}

\newtheorem{thm}{Theorem}
\newtheorem{lemma}[thm]{Lemma}
\newtheorem{prop}[thm]{Proposition}
\newtheorem{conj}[thm]{Conjecture}

\newcommand{\cumulstieltjes}[1]{{g_{#1}^c}}
\newcommand{\coeff}[2]{{F^{(#2)}_{#1}}}
\newcommand{\VVV}[2]{{V^{#1}_{#2}}}

\DeclareMathOperator{\GL}{GL}

\title[Combinatorics of Lower Terms for Moments of Zeta]{Combinatorics of Lower Order Terms in the Moment Conjectures for the Riemann Zeta Function}

\author{Paul-Olivier Dehaye}

\date{\today}                                           % Activate to display a given date or no date

\begin{document}
\input{arxiv_body.input}

\begin{abstract}
Conrey, Farmer, Keating, Rubinstein and Snaith have given a recipe that conjecturally produces, among others, the full moment polynomial for the Riemann zeta function. The leading term of this polynomial is given as a product of a factor explained by arithmetic and a factor explained by combinatorics (or, alternatively, random matrices). We explain how the lower order terms arise, and clarify the dependency of each factor  on the exponent $k$ that is  considered. 

We use extensively the theory of symmetric functions and representations of symmetric groups, ideas of Lascoux on manipulations of alphabets, and a key lemma, due in a basic version to Bump and Gamburd. Our main result ends up involving dimensions of skew partitions, as studied by Olshanski, Regev, Vershik, Ivanov and others. 

In this article, we also lay the groundwork for later unification of the combinatorial computations for lower order terms in the moments conjectures across families of $L$-functions of unitary, orthogonal and symplectic types.
\end{abstract}

\maketitle
\section{Introduction}
\subsection{History}
This paper concerns the Riemann zeta function and its moments.

In 2000, Keating and Snaith \cite{KSLfunctions,KSzeta} computed moments of characteristic polynomials of matrices in the unitary group $U(N)$ and suggested the following conjecture.
\begin{conj}[(see \cite{KSzeta})]
\label{conj.KS}
For any positive integer $k$,
\begin{equation}
\label{eqn.moment}
 \int_{0}^{T} \left| \zeta\left(\frac{1}{2} + \mathfrak{i} t\right) \right|^{2k} \ud t \sim \frac{a_k g_k}{(k^2)!}T (\log T)^{k^2},
\end{equation}
with \emph{arithmetic factor}
\begin{equation}
a_k := \prod_p \left(1-\frac{1}{p}\right)^{k^2}  {}_2F_1\left(k,k;1;\frac{1}{p}\right),
\label{eqn.ak}
\end{equation}
and \emph{combinatorial factor}\footnote{\label{footnote.rmt}This is usually referred as the \emph{random matrix theory factor}. We feel however that the adjective ``combinatorial'' is more appropriate, particularly in light of the results presented here.}
\begin{equation}
\label{eqn.gk}
g_k := \lim_{N \rightarrow \infty } \frac{\Gamma(k^2+1)}{N^{k^2}} \int_{U(N)} \left|\Lambda_g(1)\right|^{2k} \ud g =  (k^2)! \prod_{j=0}^{k-1} \frac{j!}{(j+k)!},
\end{equation}
where $\Lambda_g$ is the characteristic polynomial of a matrix $g$ distributed under Haar measure. 
\end{conj}
The normalization by $(k^2)!$ is there to guarantee that $g_k \in \mathbb{N}$. 

Their conjecture followed a few unconditional results early in the $20^\text{th}$ century (Hardy and Littlewood for $k=1$, Ingham for $k=2$) and a few conjectures in the late $20^\text{th}$ century (folklore for the exponent of $\log T$ and the general shape of the coefficient, Conrey and Ghosh \cite{MR762188} for more conjectural information about this coefficient, Conrey and Ghosh  and Conrey and Gonek \cite{MR1639551,MR1828303} for the conjectured values of $g_3$ and $g_4$). While agreeing with all the previously known results or conjectures, Keating and Snaith's contribution was to identify the sequence $g_k$ that should work: $(g_i) = (1, 1, 2, 42, 24024, 701149020, 1671643033734960,\cdots)$. This is sequence A039622 in \cite{A039622}.

Soundararajan  has shown that, assuming the Riemann Hypothesis, the exponent of $\log T$ in Equation~\eqref{eqn.moment}  is correct, at least up to an $\epsilon$ (see \cite{Sound} for the precise statement). 

The work of Keating and Snaith attracted significant interest, and led to several simplifications or new interpretations.

In 2002, the collaboration of Conrey, Farmer, Keating, Rubinstein and Snaith  \cite{CFKRS1,CFKRS2} gave a \emph{recipe} leading to a sharpening of Conjecture~(\ref{eqn.moment}) for integral $k$. This is based on earlier unconditional work of Motohashi \cite{MR2271611} sharpening the theorem of Ingham for $k=2$. Their final conjecture takes the form
\begin{conj}[(see \cite{CFKRS1}, \cite{CFKRS2})]
\label{conj.pol}
Let $k \in \mathbb{N}$. There exists a sequence $P_k$ of polynomials, of degree $k^2$ and with leading coefficient equal to $\frac{a_kg_k}{(k^2)!}$, given either implicitly by a $2k$-fold contour integral in \cite{CFKRS1} or explicitly by Theorem~1.2 in \cite{CFKRS2}, such that, for all $\epsilon >0$,
\begin{equation}
\int_0^T \left|\zeta\left(\frac{1}{2} + \mathfrak{i}t\right)\right|^{2k} \ud t = \int_0^T P_k\left(\log \frac{t}{2\pi}\right) \ud t + O\left(T^{1/2+\epsilon}\right).
\label{eqn.momentpol}
\end{equation}
\end{conj}
The information on the leading term shows that Conjecture~\ref{conj.pol} implies Conjecture~\ref{conj.KS}. 
While the leading term is simply given as a product of the $a_k$ and the $g_k$ factor, it became a major issue to understand the interaction of the arithmetic and combinatorial terms for the lower order coefficients\footnote{As far as is currently known, there is no expression for these lower order coefficients directly in terms of random matrices, thereby justifying the footnote on page~\pageref{footnote.rmt}.}.

Diaconu, Goldfeld and Hoffstein \cite{Diaconu} immediately reformulated these conjectures (or their leading order) as consequences of further conjectures in the then-emerging theory of Multiple Dirichlet Series. This is in some ways even closer to the works of Motohashi.

In 2003, Beineke and Bump \cite{BumpBeineke1,BumpBeineke2}  remarked that the CFKRS recipe could potentially be explained as constant terms of Einsenstein series on $\GL(2k)$, where a distinguished role for the subgroup $\GL(k)\times\GL(k)$ is to be expected.

In 2004,  Bump and Gamburd \cite{BumpGamburd} simplified Keating and Snaith's computations on  moments of random matrices. This work stressed the combinatorial nature of $g_k$, which has no arithmetical content.

\subsection{Motivation}
Conjecture~\ref{conj.pol} is the most precise conjecture on moments of the Riemann zeta function  that has been formulated so far (and is part of a large set of other similar conjectures on families of $L$-functions). The polynomials $P_k$ are thus very interesting, but they need to be known very precisely to be useful. Indeed, they tend to have small leading coefficient and comparatively large middle coefficients, which is problematic for testing:  for actual numerical data (\textit{i.e.}~large but finite $T$ \cite{Hiary}), it is harder to tell which term would contribute most to $P_k(\log T)$. 

The heuristics leading to Conjecture~\ref{conj.pol} are based on approximations of zeta by Dirichlet polynomials, followed by complicated combinatorial manipulations. This is similar to the earlier conjectures of Conrey and his collaborators. Unfortunately, the end result in \cite{CFKRS1} for $P_k$ is extremely implicit, given as the residue of a $2k$-fold  integral, which can be  either obtained symbolically or integrated numerically.  In either case, limitations creep in that drastically bound the size of achievable $k$.  This formula (a $2k$-fold integral) makes it also very difficult to think of the analytic continuation in $k$ of the $P_k$, which should exist (these questions are tied to bounds for the Riemann zeta function on the critical line, and eventually to the Lindel\"of Hypothesis). The end result in \cite{CFKRS2} answers some of these questions, but the computations are even more complicated and in the end useful numerically but not so illuminating. In particular, a polynomiality result for components of the $P_k$ is proved in their Theorem 1.3.  These polynomial components,  which they call $N_k(\alpha;\beta)$,  are then recovered using either interpolation or determinantal formulas.  Concern for these polynomial components end up occupying a significant part of \cite{CFKRS2}, but in fact we would argue that the $N_k(\alpha;\beta)$ are not satisfactorily explained by their paper. As evidence we will point to a very recent\footnote{In fact, a first version  of this paper (with minor differences except in the introduction and Section~\ref{sec.uniform})  preceded their preprint on the arXiv.} preprint \cite{Goulden} which aims to extend the treatment of \cite{CFKRS2} to symplectic and orthogonal families. While the arithmetic side requires little new ideas, the study of what they now call $N_\lambda(k)$ (in their Theorem~1.1) and again its polynomiality is the hardest part and requires to perform complicated calculations anew. 

Starting with the recipe of \cite{CFKRS1}, we wish to provide an alternative to the computations of \cite{CFKRS2}, avoiding entirely the use of multiple integrals. 

\subsection{Statement of the main result}
In accordance with CFKRS, we define functions $c_r(k)$ at integers $k,r\ge 0$ such that 
\begin{equation}
\label{eqn.crk}
P_k(x) = c_0(k) x^{k^2} + c_1(k) x^{k^2-1} +\cdots + c_{k^2}(k).
\end{equation}

Before we give a formal statement of our main results, we qualitatively describe them. The narrow view of our results is that the polynomial components (or in fact linear combinations of the $N_k(\alpha;\beta)$) have a combinatorial interpretation counting standard tableaux of skew shape or alternatively paths in the Young lattice. This interpretation, which is very natural to a combinatorialist, immediately leads to the polynomiality results thanks to the theory of Frobenius-Schur functions, and gives numerous efficient determinantal formulas to compute them. This interpretation will even give the analytic continuations in $k$. The broader view is that this interpretation will turn out to work in the symplectic and orthogonal cases (in the upcoming paper \emph{Combinatorics of the lower order terms in the moment conjectures for the symplectic and orthogonal families of $L$-functions}), thereby unifying our understanding of the computations needed for lower order terms across the three types of families.

We comment further on the significance of these results after stating them formally.
\begin{thm} The coefficient $c_N(k)$ satisfies the equation
\label{thm.main}
\begin{equation}
c_N(k) = \frac{1}{(k^2-N)!} \sum_{\substack{\kappa,\lambda \\ |\kappa|+|\lambda|=N}} d_{\kappa\lambda} \dim(\lambda, \complementary{k}{\kappa}),
\label{eqn.main}
\end{equation}
with the sum taken over pairs of partitions $\kappa, \lambda$ of combined weight $N$, and
\begin{itemize}
\item   $\dim(\mu,\nu)$, a nonnegative integer,  the number of standard tableaux of skew shape $\mu\setminus \nu$ (defined in Section~\ref{sec.dimension});
\item  $\complementary{k}{\kappa}$ a partition of size $k^2-|\kappa|$ (defined in Section~\ref{sec.bumpgamburd});
\item $d_{\kappa\lambda}= d_{\kappa\lambda}(k)$ a coefficient of arithmetic significance.
%\item $\omega(\kappa,\lambda)$ the parity of a permutation (also defined in Section~\ref{sec.bumpgamburd}, in Proposition~\ref{lemma.1}).
\end{itemize}
  
 To be more explicit, the $d_{\kappa\lambda}$ are given in  Proposition~\ref{prop.expansion}, Section~\ref{sec.alternate} as simple algebraic expressions (with rational coefficients depending on character tables of symmetric groups of size $N$, see the proof of Proposition \ref{prop.expansion}) in the $W_{\mu\nu}$, which are defined in Equation~\eqref{eqn.Wmunu} as 
\begin{equation*}
W_{\mu\nu} = \sum_{r \ge 1} \VVV{r}{\mu\nu}   \coeff{r}{|\mu|+|\nu|},
\end{equation*}
with the $\VVV{r}{\mu\nu}$ easily computable polynomials in $k$ (see Equation~\eqref{eqn.Vrmunu}), of degree bounded by $2r$, and $\displaystyle{\coeff{r}{n}}$ the $n^\text{th}$ Taylor coefficient of the prime zeta function at $r$ (see Section~\ref{sec.primezeta}). Moreover, the algebraic expression for the $d_{\kappa\lambda}$ only involves $W_{\mu\nu}$ of smaller total weight in $\mu$ and $\nu$ than $N$. 

%Finally, the RHS of Equation~\eqref{eqn.main} is analytic in $k$. 
\end{thm}
The (summary of the) proof of this theorem is on page~\pageref{proof.thm.main}, where references are given to the different lemmas and propositions needed. Some of the main formulas are listed in the appendix, page~\pageref{appendix}.

As will be clear from the definition of $\dim(\mu,\nu)$ and $\complementary{k}{\kappa}$, the formula would give $c_N(k) = 0$ if applied when $N>k^2$. 

The goal of this theorem is to express the $c_N(k)$ in terms of several functions that are very simple to precompute, with a clear sense of breaking down each into simpler and simpler functions and with many of them independent of $k$. We know from previous work of CFKRS that some components of the $P_k$ will not be computable exactly, and we have relegated that part to computing the Taylor coefficients of the prime zeta function, which can be done to high numerical precision. The sum giving the $W_{\mu\nu}$ is the last bit that is problematic, as it is infinite. Note however that $\coeff{r}{n}$ decreases exponentially fast in $n$ to 0, so the tail of the sum could be bounded, for instance if we only want a numerical approximation. See also Section~\ref{sec.precomputation} for more details.

Before commenting further on this theorem, we state another theorem that will clarify the dependency in $k$ of the RHS of Equation~\eqref{eqn.main}.
\begin{thm}
\label{thm.dim}
If $\kappa$ and $\lambda$ are partitions, then 
\begin{equation}
\label{eqn.dim}
\dim(\lambda, \complementary{k}{\kappa}) =  \dim(\kappa, \complementary{k}{\lambda}) = \frac{g_k B(k)}{k^2\cdot (k^2-1)\cdots(k^2-|\kappa|-|\lambda|+1)},
\end{equation}
where $B(k)$ is an explicitly computable polynomial in $k$ of degree bounded by $2(|\kappa|+|\lambda|)$ and $g_k$ is the combinatorial factor in  Conjecture~\ref{conj.KS}.
\end{thm}
The proof of this theorem is on page~\pageref{proof.thm.dim}.

We now return to the discussion of Theorems~\ref{thm.main} and \ref{thm.dim} in the context of existing literature. 
Theorem~\ref{thm.main} has a  structure similar to Theorem~1.2 of \cite{CFKRS2}: both give an explicit formula for the coefficient $c_N(k)$   as a sum over pairs of partitions of total weight $N$. In both theorems, each summand is a product of a term ``explained by arithmetic'' and another one ``explained by combinatorics'' that has a nice dependency in $k$. The polynomials $N_k(\alpha, \beta)$ in \cite{CFKRS2} entered on the combinatorial side, and this paper uses instead the function $\dim(\lambda, \complementary{k}{\kappa})$. This has several advantages:
\begin{enumerate}
\item The paper \cite{CFKRS2} offered to compute the polynomials $N_k(\alpha,\beta)$ in $k$ using either interpolation on the values of $k \times k$ determinants or via a more direct evaluation (in Equation (3.11)\footnote{Unfortunately, this formula has a typo on the right hand side: one needs to substitute $s \times s$ for $k \times k$.}  of its Theorem~3.2) as a $s \times s$ determinant, where $s$ can be bounded, after careful reading, by $s \le \frac{|\alpha|+|\beta|}{2} = \frac{N}{2}$, with $N$ the index of the coefficient $c_N(k)$ that is considered (remember from Equation~\eqref{eqn.crk} that it is meaningful to let $N$ go up to $k^2$, when $k \in \mathbb{N}$). 

Since we rely on existing theory of the Frobenius-Schur functions, we can hope for better formulas to compute the $\dim(\lambda, \complementary{k}{\kappa})$.  Indeed, the Giambelli formula in Equation~\eqref{eqn.giambelli} will give the $\dim(\lambda, \complementary{k}{\kappa})$ as determinants of size at most $\floor{\sqrt{N}} \times \floor{\sqrt{N}}$, a significant gain from a computational perspective (see also Section~\ref{sec.uniform} for more on this bound).

\item Our formula gives a concise and direct formula for the analytic continuation in $k$ (this requires Theorem~\ref{thm.dim} as well).

\item
The derivation on the combinatorial side will be unified in \emph{Combinatorics of the lower order terms in the moment conjectures for the symplectic and orthogonal families of $L$-functions} with the derivation for other families of $L$-functions.

\item On the arithmetic side, while \cite{CFKRS2} gives a succession of steps, we have figured out the exact final formulas.
 Numerically, this is an advantage: the non-exact part of the computation is limited to the evaluation of the Taylor coefficients $\coeff{r}{n}$, and then the truncation of the sum for $W_{\mu\nu}$. It should be easier to achieve better precision, and  would be conceivable to use interval arithmetic to bound exactly the error actually made.

\item In \cite{Rub1}, Hiary and Rubinstein obtained uniform bounds on the coefficients $c_N(k)$. For this, they had to bound the arithmetic side and the combinatorial side separately. We did no 
work in this direction, but comment on connections in Section~\ref{sec.uniform}.
\end{enumerate}

\subsection{Structure of the paper}
This paper requires a variety of very different tools. In Section~\ref{sec.essentials}, we present the background needed on partitions, symmetric functions, alphabet manipulations and power series expansions of the Riemann zeta function and the prime zeta function. 
In Section~\ref{sec.recipe}, we present the beginning of the CFKRS recipe in details, but formulated using symmetric functions. In Section~\ref{sec.bumpgamburd}, we introduce the substitute to the $2k$-fold integrals of CFKRS, a lemma due to Bump and Gamburd in its basic form.  In Section~\ref{sec.alternate}, we compute the final expression handed to us by the recipe, using that lemma. Section~\ref{sec.explicitexpansion} contains the manipulations that allow us to transfer Dirichlet polynomials computations in the language of symmetric functions. In Section~\ref{sec.thm3}, we present the proof of Theorem~\ref{thm.dim}. Section~\ref{sec.corollaries} contains the various corollaries to the main results, and in particular the rederivation of the Keating-Snaith leading coefficient. In Section~\ref{sec.precomputation}, we explain how to compute the constants involved in our derivation. Due to the complexity of the formulas obtained, we have decided to add a  table of the main formulas and definitions in the appendix, page~\pageref{appendix}.

\begin{acknowledgements}
The author wishes to thank many people for useful discussions during the (long) period that led to this paper and in particular Profs.~Bump, Farmer and Nikeghbali. In addition, Profs.~Conrey and Kowalski read preliminary drafts and provided feedback and corrections. Many tests and computations were performed using  the computer algebra package \textsf{sage} \cite{sage}.
\end{acknowledgements}

\section{Some Essentials}
\label{sec.essentials}
\subsection{Essentials in the theory of partitions}
We rely almost exclusively on definitions and objects introduced in the first chapter of \cite{Macdonald}. We briefly present these definitions. 
\label{sec.partitions}
\subsubsection{Definition and easy invariants}
An (integer) partition $\lambda = ( \lambda_1,\cdots,\lambda_n) $ is a finite weakly decreasing sequence of non-negative integers, called \emph{parts}. We define the \emph{weigh}t $|\lambda|$ of $\lambda$ to be the sum $\sum_j \lambda_j$. If this weight is $n$, we also use the notation $\lambda \vdash n$. The length $l(\lambda)$ of $\lambda$ is the maximal $l$ such that $\lambda_{l} \ne  0$. An empty sequence $()$ is thus also a partition, of length 0 and weight 0. We denote this empty partition by $\phi$. 

We freely think of partitions in terms of \emph{Ferrers} (also called \emph{Young}) \emph{diagrams}. See Figure~\ref{fig.complementary}, page~\pageref{fig.complementary} for some examples. The Ferrers diagram of a partition is constituted of \emph{boxes}. For each box $\square$, we call the \emph{hook} of $\square$ the set consisting of the boxes exactly to the right or  exactly below $\square$, and $\square$ itself. The \emph{hook length} of $\square$ is the cardinality of that set.

The \emph{conjugate} of $\lambda$ is the partition obtained by flipping the diagram of $\lambda$ along the main diagonal. It is denoted $\lambda^t$.  

Given a partition $\lambda$, we can naturally associate a conjugacy class of permutations in  $\mathcal{S}_{|\lambda|}$ by considering that the $\lambda_i$ give the lengths of the cycles of permutations. We define $z_\lambda$ to be the order of the centralizer of an element of that conjugacy class. It is immediate to compute \begin{equation}
z_\lambda = \prod_{j\ge 1}  \left(j^{M_j(\lambda)}M_j(\lambda)!\right),\label{eqn.zlambda}
\end{equation}
where  $M_j(\lambda)$ counts the number of parts  of length $j$ (hence almost all of those will be 0). 

We also need what we call \emph{vectorized partitions}. Given a partition $\lambda$ of length $l\le L$, we sometimes want to consider the partition as a vector of length $L$.  In that case we use the symbol  $\lambda$ for the vector $(\lambda_1,\cdots,\lambda_l,0,\cdots,0)$, \textit{i.e.}~we pad the sequence of $\lambda$ with $L-l$ zeroes. We say that $\lambda$ is \emph{vectorized to length $L$}.  We can of course add any two partitions vectorized to the same length: we do this componentwise, as vectors. We let $\rho_L$ be the partition $(L-1,\cdots,1)$ (with triangular Ferrers diagram), vectorized to length $L$ (by appending one 0). 

In addition to conjugacy classes in symmetric groups, partitions are used to index many objects of combinatorial interest, such as symmetric function bases or characters of symmetric groups. We will explain some of this in Section~\ref{sec.symmetric}.

\subsubsection{The dimension of a (skew) partition}
\label{sec.dimension}
We define the \emph{Young lattice} as a graph $(V,E)$. 
Let $V_n$ be the set of Ferrers diagrams of weight $n$. Set $V$ to be $\cup_{n \ge 0} V_n$. The (undirected) edges $E$ only connect Ferrers diagram that differ by the addition or removal of exactly one box. Edges can thus only link partitions whose weight differ by 1. 

Given two partitions $\kappa, \lambda$, of weight $k$ and $l$ respectively, we let $\dim(\kappa,\lambda)$ to be the number of paths in the Young lattice of length exactly $l-k$ connecting $\kappa$ and $\lambda$. In particular, this will be 0 if $k>l$, or if the Ferrers diagram of $\kappa$ is not contained inside the Ferrers diagram of $\lambda$. 

There exist convenient formulas to compute $\dim(\kappa,\lambda)$, and these will be presented in Section~\ref{sec.FS}.
\subsection{Essentials in symmetric function theory}
We present here the concepts of symmetric functions. We follow entirely the conventions of Macdonald \cite{Macdonald}, except that we also define the symmetric polynomials $\sympsymb{0}$.
\label{sec.symmetric}
\subsubsection{The ring of symmetric polynomials}

Let $\XXX = \{x_1,\cdots,x_n\}$. Consider the ring $\mathbb{C}[[x_1,\cdots,x_n]]$ of power series in the variables $\XXX$. The symmetric group $\mathcal{S}_n$ acts on this ring by permuting variables, and a power series is symmetric if it is invariant under this action. The symmetric power series form a subring 
\begin{equation}
\bar{\Lambda}_n = \bar{\Lambda}_n(\XXX) = \mathbb{C}[[x_1,\cdots,x_n]]^{\mathcal{S}_n}.
\end{equation}
We also set $\Lambda_n = \bar{\Lambda}_n \cap \mathbb{C}[\XXX].$ The ring $\Lambda_n$ is a graded ring: we have 
\begin{equation}
\Lambda_n = \underset{k \ge 0}{\oplus} \Lambda_n^{(k)},
\end{equation}
with each component consisting of polynomials of degree $k$ (together with the 0 polynomial).

We will use most often elements of $\bar{\Lambda}_{n}$, \textit{i.e.}~not just polynomials but infinite series of elements of $\Lambda_n$. We call these \emph{symmetric functions}.
\subsubsection{Bases}
 \label{sec.bases}
Combinatorialists have studied several different bases of polynomials in $\Lambda_n$. We introduce them now.

For each $\alpha = (\alpha_1,\cdots,\alpha_n) \in \mathbb{N}^n$, we set $x^\alpha := x_1^{\alpha_1}\cdots x_n^{\alpha_n}$.  Given a partition $\mu$ such that $l(\mu) \le n$, we let the \emph{monomial symmetric polynomials} be
\begin{equation}
\symm{\mu}{\XXX} := \sum_{x \in \XXX} x^\alpha,
\end{equation}
where the sum runs over all the distinct permutations $\alpha$ of $\lambda=(\lambda_1,\cdots,\lambda_n)$. The set of $\symmsymb{\lambda}$ such that $l(\lambda) \le n $ and $|\lambda|=k$ forms a $\mathbb{Z}-$basis of $\Lambda^{(k)}_n$.

For each $r \ge 0$, we define the \emph{complete symmetric polynomial} $\symhsymb{r}$ as
\begin{equation}
\symhsymb{r} := \sum_{\lambda \vdash r} \symmsymb{\lambda}.
\end{equation} 
In particular, we have $\symhsymb{0} =1 $. We also set $\symhsymb{\mu}:= \prod_{i=1}^{l(\mu)} \symhsymb{\mu_i}$. In particular, we have $\symhsymb{\phi} =1 $.

For each $r\ge 0$, we also define the \emph{power symmetric polynomial} $\sympsymb{r}$ as 
\begin{equation}
\symp{r}{\XXX} := \sum_{x \in \XXX} x^r.
\end{equation}
In particular\footnote{This is true even if some $x_i$ is 0! Actually, we always avoid directly setting any $x_i$ to 0, but we let them tend to 0. The value of $\symp{r}{\XXX}$ is then less shocking.}, we have $\symp{0}{\XXX} = |\XXX|$. We also set $\sympsymb{\mu} := \prod_{i=1}^{l(\mu)} \sympsymb{\mu_i}$. In particular, we\footnote{Macdonald does not define $\sympsymb{0}$, possibly because of the potential confusion arising from $\sympsymb{0} \ne \sympsymb{\phi}$. These will actually turn out to be quite helpful to us: they allow to temporarily hide away the dependency of expressions on the size of the set $\XXX$.} have $\sympsymb{\phi} = 1$.

The last set of polynomials is the \emph{Schur polynomials} $\symssymb{\mu}$ for $\mu$ running through partitions.
If $l(\mu) > n$, we set $\syms{\mu}{\XXX} = 0$ and otherwise 
\begin{equation}
\label{eqn.schur}
\syms{\mu}{\XXX} := \frac{\left| x_i^{\mu_j+n-j}\right|}{\left| x_i^{n-j} \right|},
\end{equation} 
after vectorizing $\mu$ to length $n$ if needed. Note that $\symssymb{\phi} =1$.

 \subsubsection{Transitions between bases}
 
 \label{sec.transition}
 For each pair of basis of $\Lambda$, there exists a transition matrix from one to the other. They have long been known, and Macdonald devotes a whole chapter to them \cite[Chapter~6]{Macdonald}.
 
For each partition $\lambda$ there exists a character $\chi^\lambda$ of the symmetric group $ \mathcal{S}_{|\lambda|}$. These characters are class functions, and thus only dependent on the cycle type of the conjugacy class considered.  Denote the value of that character on the conjugacy class of cycle type given by the partition $\mu$ as $\chi^\lambda(\mu)$.   This coefficient is known to be an integer, and the dimension of the character $\chi^\lambda$, \textit{i.e.}~$\chi^\lambda(\text{id})=\chi^\lambda((1,\cdots,1))$, equals $\dim(\phi,\lambda)$. The whole theory is beautifully introduced in \cite{Sagan}. 

It turns out these character values also serve as transition coefficients between the $\sympsymb{\mu}$ and the $\symssymb{\lambda}.$ Indeed, we have\footnote{Remember (from Equation~\eqref{eqn.zlambda}) that $z_\mu$ is the order of the centralizer of any permutation of cycle type $\mu$, and that therefore the conjugacy class of cycle type $\mu$ is of order $\frac{|\mu|!}{z_\mu}$.}
\begin{equation}
\label{eqn.transitionschur}
\symssymb{\lambda} = \sum_{\mu \vdash |\lambda|} \frac{1}{z_\mu} \chi^{\lambda}(\mu) \sympsymb{\mu}
\end{equation}
and
\begin{equation}
\label{eqn.transitionpower}
\sympsymb{\mu}=\sum_{\lambda\vdash |\mu|} \chi^{\lambda}(\mu) \symssymb{\lambda}.
\end{equation}
The only other transition we will need (actually a special case of the first transition presented above) is 
\begin{equation}
\symh{n}{\AAA} = \sum_{\kappa \vdash n} \frac{1}{z_\kappa} \symp{\kappa}{\AAA}.
\end{equation}

 \subsubsection{Pieri rule}
This rule gives a multiplication formula in $\Lambda$ for a Schur polynomial $\symssymb{\kappa}$ and the power polynomial $\sympsymb{1}$:
\[
\symssymb{\kappa} \sympsymb{1} = \sum_{\kappa^+} \symssymb{\kappa^+},
\]
where $\kappa^+$ runs through partitions $\mu$ such that $\dim(\kappa,\mu)$ equals 1. By repeated application of this rule, we see that 
\[
\symssymb{\kappa} (\sympsymb{1})^r = \sum_{\mu\vdash |\kappa|+r} \dim(\kappa,\mu) \symssymb{\mu}.
\]

 \label{sec.pieri}
  
\subsection{Essentials on alphabets}
We only use a very elementary version of Lascoux' theory of alphabets. This paper could almost surely be simplified by using more of that theory. 
\subsubsection{Definition}
To us, alphabets will be finite sets of indeterminates. Sometimes, indeterminates might happen to take the same value, but they will always be different indeterminates. If $\DDD$ is an alphabet, its cardinality is thus constant, and denoted by $|\DDD|$.  

We distinguish some special alphabets: $\One:= \{1\}$ and $\Zero := \{0\}$, both of cardinality 1.

\subsubsection{Notational shortcut for alphabets}
We will use alphabets constructed on three main alphabets: $\AAA$, $\BBB$ and $\DDD$. For simplicity of notation, we index elements of $\AAA$ by $\alpha$, for $\BBB$ by $\beta$ and for $\DDD$ by $\shift$. In the same spirit, we write 
 \[
\prod_{\shift} f(\shift) \text{ for } \prod_{\shift \in \DDD} f(\shift).
\]

Note also that there will be significant advantage in using the notation  $\prod_{\shift} f(\shift)$ instead of  $= \prod_{i=1}^{|\DDD|} f(\shift_j)$ for instance: the latter notation imposes an (unnatural) order on the elements of $\DDD$, while the former  localizes any dependency in the size of the alphabet, and spares the introduction of dummy indices.
\subsubsection{Basic operations on alphabets} 
%We have defined alphabets as sets of indeterminates. What operations could we possibly want to do with them? Given one alphabet $\{\alpha\}$, we could consider for instance $\{-\alpha\}$. Or given a pair of alphabets, $\{\alpha\}, \{\beta\}$, we could consider the pairwise product $\{\alpha \beta\}$, which is also a set of (more complicated) indeterminates. This is essentially what we will do, but we will look at it from a more powerful angle. 

%Alphabets will always end up in this paper as arguments to symmetric functions. For instance, we write $\symp{n}{\AAA}$ for $\sum_{\alpha} \alpha^n$.  For each assignment of values to the indeterminates in $\AAA$, we thus obtain an element of $\Lambda^*$ (an evaluation map). 

%Assume for a second that we have a fixed assignment of values to the indeterminates. 
%We could define several operations on the \emph{set} of indeterminates $\AAA$, or even pairs of \emph{sets} $\AAA, \BBB$. We know that for each assignment of values to the indeterminates, this operation on sets would have to commute with evaluation and an operation on $\Lambda$ itself. There exists Each such operation commutes with an operation on $\Lambda$ and evaluation. 
%
%This is where alphabets become especially useful: we c can define several operations on alphabets, by considering their action on the set of indeterminates. Instead, we prefer to consider these operations consider these operations on alphabets , by giving their action on the generators $\symp{m}{\cdot}$ of the ring of symmetric functions. 
Given an alphabet $\AAA$, or maybe a pair of alphabets $(\AAA,\BBB)$, one can define another alphabet $\XXX$ by performing some operation on the elements of the starting alphabet(s). We present a few constructions.

\paragraph{Repetition} Given an alphabet $\AAA$ and $k\in \mathbb{N}$, we define $k \cdot \AAA$ by taking the union of $k$ disjoint copies of $\AAA$. By extension, for any $k \in \mathbb{C}$, we can define a new ``abstract'' alphabet by
\begin{equation}
\symp{m}{k \cdot \AAA} := k \symp{m}{\AAA}.
\end{equation}
\paragraph{Pairwise addition} Given two alphabets $\AAA = \{\alpha\}$ and $\BBB = \{\beta\}$, we define $\AAA \calplus \BBB := \{\alpha+\beta: \alpha \in \AAA,\beta \in \BBB\}$. This implies 
\begin{equation}
\symp{m}{\AAA\calplus\BBB} =  \sum_{u\ge 0}\binom{m}{u} \symp{u}{\AAA}\symp{m-u}{\BBB}.
\label{eqn.pairwiseaddition}
\end{equation}
We have for instance $\symp{0}{\AAA \calplus \BBB} = \symp{0}{\AAA} \symp{0}{\BBB} = |\AAA||\BBB|$ and 
$\symp{1}{\AAA \calplus \BBB} = |\AAA| \symp{1}{ \BBB} + |\BBB| \symp{1}{ \AAA}$. Note how $\sympsymb{0}$ turns out to be useful to obtain uniform formulas.

This definition already includes the notion of translates of alphabets. Indeed, given an alphabet $\AAA = \{\alpha\}$ and an element $r \in \mathbb{C}$, we have $\AAA \calplus \{r\} =\{\alpha +r \}$.
\paragraph{Scaling}
We can also scale alphabets by a complex number. We write this  $\phi_r(\AAA) := \{r \alpha:\alpha \in \AAA\}$, which implies
\begin{equation}
\symp{m}{\phi_r(\AAA)} = r^m \symp{m}{\AAA}.
\end{equation}
A special case occurs when $r = -1$, where we use the shortcut
$ \neg{\AAA} := \phi_{-1}(\AAA). $

We can also define scalings by a partition $\kappa$:
\begin{equation}
\phi_\kappa(\AAA) := \phi_{\kappa_1}(\AAA) \calplus \cdots\calplus \phi_{\kappa_{l(\kappa)}}(\AAA).
\end{equation}

Despite these simple definitions, it already becomes harder (but will be useful) to compute $\symp{m}{\phi_\kappa(\AAA)}$.  Let $l = l(\kappa)$. Then,
\begin{align}
\symp{m}{\phi_\kappa(\AAA)} &= \symp{m}{\phi_{\kappa_1}(\AAA)\calplus \cdots \calplus \phi_{\kappa_{l(\kappa)}}(\AAA)}
\nonumber\\
& =  \sum_{i_1,\cdots,i_l} \binom{m}{i_1,\cdots,i_l} \symp{i_1}{\phi_{\kappa_1}(\AAA)}\cdots  \symp{i_l}{\phi_{\kappa_l}(\AAA)}\nonumber\\
& =  \sum_{i_1,\cdots,i_l} \binom{m}{i_1,\cdots,i_l} \kappa_1^{i_1} \cdots \kappa_l^{i_l}  \symp{i_1}{\AAA}\cdots  \symp{i_l}{\AAA}\nonumber\\
\nonumber& =  \sum_{\mu\vdash m } \binom{m}{\mu} \symm{\mu}{\kappa}  \symp{\mu}{\AAA} \symp{0}{\AAA}^{l(\kappa)-l(\mu)}\\
& =  \sum_{\mu\vdash m } \binom{m}{\mu} \symm{\mu}{\kappa}  \symp{\mu}{\AAA} |\AAA|^{l(\kappa)-l(\mu)},
\label{eqn.pmphik}
\end{align}
where $\binom{m}{\mu} := \binom{m}{\mu_1,\cdots,\mu_{l(\mu)}}$ is a multinomial coefficient.
The first equality follows from the definition of $\phi_\kappa(\AAA)$, the second from iterating the definition of pairwise addition, the third from the definition of $\symp{l}{\phi_r(\AAA)}$ and the fourth from the definition of $\symm{\mu}{\kappa}$ and re-sorting of the vector $\vec{i}$ into the partition $\mu$. The fifth uses the definition of $\symp{0}{\AAA}$. 

It follows immediately from this that we have $\symp{0}{\phi_\kappa(\AAA)} = |\AAA|^{l(\kappa)}$.

\paragraph{Pairwise multiplication} Similarly to pairwise addition, we define pairwise multiplication by
\begin{equation}
\AAA \caltimes \BBB := \{ \alpha\cdot \beta : \alpha \in \AAA, \beta \in \BBB \},
\end{equation}
which implies \begin{equation}
\symp{m}{\AAA \otimes\BBB} = \symp{m}{\AAA} \symp{m}{\BBB}.
\end{equation}
Of course, scaling is a special case of pairwise multiplication. Indeed, $\phi_r(\AAA) = \{r\} \caltimes \AAA$. In general, we have $\symp{0}{\AAA \caltimes \BBB}  = \symp{0}{\AAA} \symp{0}{\BBB}  = |\AAA|\cdot |\BBB|$.
\paragraph{Elementwise exponentiation} Let $X \in \mathbb{R}, X>0$. Given an alphabet $\AAA = \{\alpha\}$, we set $X^\AAA := \{ X^\alpha\}$. We immediately get:
\begin{equation}
\symp{m}{\calxxp{\AAA}}  := \sum_{u\ge 0} \frac{m^u (\ln X)^u}{u!} \symp{u}{\AAA}.
\label{eqn.exponential}
\end{equation}
This verifies $\symp{0}{X^\AAA} = \symp{0}{\AAA} = |\AAA|$. 

The definitions guarantee that 
\begin{equation}
\symp{m}{\calxxp{\AAA \calplus \BBB}} = \symp{m}{\calxxp{\AAA} \caltimes \calxxp{\BBB}} = \symp{m}{\calxxp{\AAA}}\cdot \symp{m}{\calxxp{\BBB}} .
\end{equation}
\paragraph{Resolvent} Finally, we define 
\begin{equation}
\Delta(\AAA;\BBB) := \prod_{\alpha\beta} (\alpha-\beta) = \prod_{\gamma \in \AAA \calplus \negBBB} \gamma,
\end{equation}
which is sometimes called the resolvent.
%The definition given here requires that $\AAA$ and $\BBB$ be given as sets of indeterminates.

\subsection{Essentials in number theory}
\subsubsection{The Riemann zeta function}
For completeness, we redefine the Riemann zeta function 
\[
\zeta(s) = \sum_{n \ge 1} \frac{1}{n^s}      \quad  \Re s> 1.
\]
We are actually interested in its power series expansion around 1, once the polar part is removed:
\[
\zeta(1+s) = \frac{1}{s} + \sum_{n=0}^\infty (-1)^n \frac{\gamma_n }{n!} s^n= \frac{1}{s}\left(1 + \sum_{n=0}^\infty (-1)^{n} \frac{\gamma_n }{n!} s^{n+1}\right),       
\]
where $\gamma_i$ are the Stieltjes constants and thus $\gamma_0$ is the Euler constant.

We will also need the expansion 
\begin{equation}
\log\left(s \cdot \zeta(1+s)\right) = -\sum_{n=1}^\infty \frac{(-1)^n}{n!} \cumulstieltjes{n} s^n,
\end{equation}
with $\cumulstieltjes{n}$ the \emph{Stieltjes cumulants}, as defined by Voros in \cite[p.~25]{VorosBook}. These coefficients  are more convenient for us, and are also more directly linked with the Riemann zeroes  than the coefficients $\gamma_n$ (\cite[p.~70]{VorosBook} and \cite[p.~676]{VorosOriginal}).

The zeta function satisfies a functional equation $\zeta(s) = \chi(s) \zeta(1-s)$, with $\chi(s) = 2^s \pi^{s-1} \sin \left(\frac{\pi s}{2}\right) \Gamma(1-s)$
and more importantly to us
\begin{equation}
\chi(s) =  \left( \frac{t}{2\pi} \right)^{\frac12-s} e^{\mathfrak{i}t +\pi \mathfrak{i}/4} \left(1+O\left(\frac{1}{t}\right)\right),
\label{eqn.chi1}
\end{equation}
for $s$ in a vertical strip in the complex plane. This leads to the relation
\begin{equation}
\chi\left(\frac12 + \mathfrak{i}t\right)^{-k} \prod_{\substack{\beta\\|\BBB|=k}} \chi\left(\frac12 + \mathfrak{i}t +\beta\right) = \left(\frac{t}{2\pi}\right)^{-\symp{1}{\BBB}} \left(1+O\left(\frac{1}{t}\right)\right),
\label{eqn.chi2}
\end{equation}
also for $s$ in a vertical strip.

\subsubsection{The prime zeta function} The prime zeta function is much less studied than the Riemann zeta function. Its definition is 
\[
\primezeta(s) = \sum_{p} \frac{1}{p^s}           \quad \Re s>1,
\]
which admits a meromorphic continuation to any simply connected domain $\Omega$ sandwiched between 
$ \{s : \Re s > 1\} \subset \Omega \subset \{s : \Re s > 0\}$ and such that $\Omega$ does not contain any $x$ with $nx$ the pole or a zero of $\zeta$, for a squarefree $n$ \cite{Froberg}. In any case, we will only need $\primezeta(s)$ for $\Re s >1$, at worst approaching $1$. 

The prime zeta function is much more adapted to the computations that will follow. In fact, we need its Taylor coefficients at various integers. 
When $r>1$, let 
\begin{align}
\primezeta(r+s) &=:  \sum_{n\ge 0} \coeff{r}{n} s^n,\label{eqn.coeffrn}
\end{align}
with $ \coeff{r}{0}=\primezeta(r)$.
% , $\sum_{i \ge 1} c_i^{(r)} = 0$.

When $r=1$, we would like to do the same, but $\primezeta(s)$ is not meromorphic at $1$. The singularity can easily be characterized. This uses a M\"obius inversion \cite{Froberg}:
\begin{equation}
\label{eqn.moebius}
\primezeta (s) = \sum_{n\ge 1} \frac{ \mu(n)}{n}  \log \zeta(ns) \quad \Re s > 1,
\end{equation}
which seems to have already been known to Glaisher (1891).  In a neighborhood of 1, we can use this to study the singularity:
\begin{equation}
\label{eqn.coeffr1}
\primezeta (1+s)  = -\log(s) -\sum_{n=1}^\infty \frac{(-1)^n}{n!} \cumulstieltjes{n} s^n+ \sum_{n\ge 2} \frac{\mu(n)}{n}  \log \zeta(n+ns).
\end{equation}

This leads us to define the coefficients 
\begin{equation}
\label{eqn.coeff1n}
\primezeta(1+s) =:  -\log (s) + \sum_{n\ge 0} \coeff{1}{n} s^n,
\end{equation}

with the observation that $\coeff{1}{0}= \sum_{n\ge 2} \mu(n)  \frac{\log \zeta(n)}{n} = B_1-\gamma_0 (\approx -0.315718452$), with $B_1$ the Mertens constant.
\label{sec.primezeta}

\section{Following the CFKRS Recipe, at least initially}
\label{sec.recipe}
CFRKS start with (an expression similar to) 
\begin{equation}
\label{eqn.shifts}
\int_0^T \left| \zeta\left(\frac{1}{2}+\mathfrak{i}t\right)\right|^{2k} \ud t = \lim_{\DDD \rightarrow (2k) \cdot \Zero} \underbrace{\int_0^T \chi\left(\frac12 + \mathfrak{i}t\right)^{\!\!-k} \prod_{\shift}\zeta\left(\frac12 + \shift + \mathfrak{i}t \right)  \ud t}_{(*)},
\end{equation}
where the functional equation is $\zeta(s) = \chi(s)\zeta(1-s)$ and $\DDD \rightarrow  (2k) \cdot \Zero$ means that we are letting each of the $2k$  indeterminates of $\mathcal{D}$ go to 0.

The recipe is composed of three main steps. As CFKRS insist \cite[p.~53]{CFKRS1}, these are not individually heuristics: none of those steps can be properly justified, and some in fact give diverging expressions. Some large terms are added while some others are then thrown out. However, the conjecture is that the composition of the three steps gives the correct final formula.  Indeed, it agrees with known theorems for special cases, and is otherwise reasonable and matches experimental data. We name the three steps as follows:
\begin{enumerate}
\item approximate functional equation step;
\item $\chi$-selection step;
\item diagonal selection step.
\end{enumerate}
We will always indicate a traditional equality with the traditional $=$ sign. However, we also use  the symbol $\Rightarrow$ to indicate one of the three unjustifiable steps above. 

We first state the approximate functional equation \cite{MR882550}. Let $0\le \Re s = \sigma \le 1$, and $t>C>0$ with $2\pi MN = t$. Then, uniformly in $\sigma$,
\begin{equation}
\zeta(s) = \sum_{n=1}^N \frac{1}{n^s} + \chi(s) \sum_{n=1}^M \frac{1}{n^{1-s}} +O(\text{error}).
\label{eqn.approx}
\end{equation} 
The \textbf{approximate functional equation step} replaces each $\zeta\left(\frac{1}{2} + \shift + \mathfrak{i}t\right)$ factor in the RHS of \eqref{eqn.shifts} by 
\begin{equation}
\sum_{n\ge1} \frac{1}{n^{\frac{1}{2} + \shift + \mathfrak{i}t}} + \chi\left(\frac{1}{2} + \shift + \mathfrak{i}t\right) \sum_{n\ge1} \frac{1}{n^{\frac{1}{2} - \shift -\mathfrak{i}t}},
\label{eqn.approx2}
\end{equation}
which is cause for serious concern: neither of the two sums converge when the shift $\shift$ is too small. Nevertheless, CFKRS carry on and replace each factor in $(*)$ by \eqref{eqn.approx2}. 

Explicitly, we obtain for $(*)$:
\begin{align}
\nonumber
(*)& =\int_0^T \chi\left(\frac12 + \mathfrak{i}t\right)^{\!\!-k} \prod_{\shift} \zeta\left(\frac12 + \shift + \mathfrak{i}t \right)  \ud t\\
&\Rightarrow
\int_0^T \chi\left(\frac12 + \mathfrak{i}t\right)^{\!\!-k} \prod_{\shift}\left[ 
\sum_{n\ge1} \frac{1}{n^{\frac{1}{2} + \shift + \mathfrak{i}t}} + \chi\left(\frac{1}{2} + \shift + \mathfrak{i}t\right) \sum_{n\ge1} \frac{1}{n^{\frac{1}{2} - \shift -\mathfrak{i}t}} \right] \ud t.
\label{eqn.after1}
\end{align}
A na\"ive expansion of this last product over $\delta$ leads to a sum of $2^{2k}$ terms. The alphabets $\AAA$ and $\BBB$ serve to encode which of the two terms is chosen for each factor. The sum is thus over the set partitions of $\DDD = \left\{\delta\right\}$. 
\begin{equation*}
\text{Eqn.}~\eqref{eqn.after1}= \int_0^T \chi\left(\frac12 + \mathfrak{i}t\right)^{\!\!-k} \!\!\! \sum_{\substack{\AAA,\BBB \text{ s.t.}\\ \AAA \cup \BBB = \DDD \\\AAA \cap \BBB = \{\}}}
\prod_\alpha \sum_{n\ge1} \frac{1}{n^{\frac{1}{2} + \alpha + \mathfrak{i}t}} \cdot
\prod_\beta \left( \chi\left(\frac{1}{2} + \beta + \mathfrak{i}t\right)  \sum_{n\ge1} \frac{1}{n^{\frac{1}{2} - \beta - \mathfrak{i}t}}\right) 
\ud t.
\end{equation*}
We now use information about $\chi$ contained in Equation~\eqref{eqn.chi1}. For high value of $t$, on the critical strip, $\chi(t)$ will oscillate very rapidly (due to the $t^{\mathfrak{i}t}$ term in Equation~\eqref{eqn.chi1}).  This will lead to a negligible contribution once integrated over $t$. Therefore, the \textbf{$\chi$-selection step} imposes the further restriction that $|\BBB|=k$. Indeed, those terms will dominate, due to the simplification of Equation~\eqref{eqn.chi2}:
\begin{equation}
\Rightarrow \int_0^T   \sum_{\substack{\AAA,\BBB \text{ s.t.}\\ \AAA \cup \BBB = \DDD \\|\AAA|=|\BBB|=k}}\left\{ \left( \frac{t}{2\pi}\right)^{-\symp{1}{\BBB}}
\prod_\alpha \sum_{n\ge1} \frac{1}{n^{\frac{1}{2} + \alpha + \mathfrak{i}t}} \cdot
\prod_\beta  \sum_{m\ge1} \frac{1}{m^{\frac{1}{2} - \beta - \mathfrak{i}t}}\right\}
\ud t,
\label{eqn.after2}
\end{equation}
where we have used the dummy variables $m$ and $n$ to prepare for what follows.
We now expand the products over $\alpha$ and $\beta$ in the most na\"ive way possible. This leads to the choice of $n_1, \cdots, n_k$ associated to $\alpha_1,\cdots,\alpha_k$ and  $m_1, \cdots, m_k$ associated to $\beta_1,\cdots,\beta_k$.  Each summand in the full expansion will be of the form
\begin{equation}
c(\vec{n},\vec{m},\vec{\alpha},\vec{\beta};t):= \frac{\prod_{i=1}^k m_i^{\beta_i}}{\prod_{i=1}^k n_i^{\alpha_i} } \cdot \frac{1}{(\prod_{i=1}^k n_i m_i)^{\frac{1}{2}}} \cdot \left(\frac{\prod_{i=1}^k m_i}{\prod_{i=1}^k n_i }\right)^{\mathfrak{i}t}.
\end{equation}
It is immediate that
\begin{equation}
\int_0^T c(\vec{n},\vec{m},\vec{\alpha},\vec{\beta};t) \ud t = 
\begin{cases}
 \frac{T}{M}  \frac{\prod_{i=1}^k m_i^{\beta_i}}{\prod_{i=1}^k n_i^{\alpha_i} } 
 & \text{if } M:= \prod_i m_i = \prod_i n_i;\\
o(T) & \text{otherwise.}
\end{cases}
\end{equation}
This is the ``justification'' \cite[p.~35]{CFKRS1} for the \textbf{diagonal term selection}, which will only keep the terms in \eqref{eqn.after2} not in $o(T)$ above. 

The expression obtained so far, at the stage \eqref{eqn.after2}, was of the form 
\begin{equation}
 \int_0^T   \sum_{\substack{\AAA,\BBB \text{ s.t.}\\ \AAA \cup \BBB = \DDD \\|\AAA|=|\BBB|=k}} \left( \frac{t}{2\pi}\right)^{-\symp{1}{\BBB}}
f(\AAA;\BBB)\ud t,
\label{eqn.after2bis}
\end{equation}
where the crucial feature of  $f$ is that it is symmetric in $\AAA$ and $\BBB$, in the sense that $f\left(\alpha_1,\cdots,\alpha_k; \beta_1,\cdots,\beta_k\right)  = f\left(\alpha_{\sigma(1)},\cdots,\alpha_{\sigma(k)}; \beta_{\sigma'(1)},\cdots,\beta_{\sigma'(k)}\right)$ for any  $(\sigma,\sigma') \in  \mathcal{S}_k\times  \mathcal{S}_k$, if $\AAA = \{\alpha\} $ and $\BBB = \{\beta\}$.
In a way that will be made much more explicit in Section~\ref{sec.explicitexpansion}, the diagonal selection step transforms the combinatorial structure of \eqref{eqn.after2bis} and thus also \eqref{eqn.after2} into an expression of the form
\begin{equation}
\Rightarrow  \int_0^T   \sum_{\substack{\AAA,\BBB \text{ s.t.}\\ \AAA \cup \BBB = \DDD \\|\AAA|=|\BBB|=k}} \left( \frac{t}{2\pi}\right)^{-\symp{1}{\BBB}}
\frac{\tilde{f}(\AAA;\BBB)}{\Delta(\AAA;\BBB)}\ud t,
\label{eqn.after3}
\end{equation}
where $\tilde{f}$ has the same symmetries as $f$. Before describing in more details in Section~\ref{sec.explicitexpansion} the transition from $f$ to $\tilde{f}$, we will now present the advantages of an expression of the form \eqref{eqn.after3}. 

\section{Laplace Expansions and the Bump-Gamburd Lemma}
\label{sec.bumpgamburd}
Imagine we could obtain the following equality for the numerator in Expression~\eqref{eqn.after3}:
\begin{equation}
\left(\frac{t}{2\pi}\right)^{-\symp{1}{\BBB}} \tilde{f}(\AAA;\BBB) = \sum_{\kappa,\lambda} b_{\kappa,\lambda} \syms{\kappa}{\AAA} \syms{\lambda}{\BBB},
\label{eqn.summand}
\end{equation} 
\textit{i.e.}~express this numerator as an infinite sum of products of two Schur functions $ \syms{\kappa}{\AAA} \syms{\lambda}{\BBB}$, elements of one of the most natural bases for expressions displaying this $\mathcal{S}_k\times \mathcal{S}_k$ symmetry. Obviously, the $b_{\kappa,\lambda}$ would still depend on $t$. We would then have, in \eqref{eqn.after3}, to sum \eqref{eqn.summand} over balanced disjoint set partitions $\AAA \cup \BBB = \DDD$. However we do know that, once this summing is done, we have to recover an expression with the full symmetry of \eqref{eqn.shifts}, \textit{i.e.}~a symmetry associated to permutations of the variables of $\DDD$, the full group $\mathcal{S}_{2k}$. 

Fortunately, a clever use of the Laplace expansion, due to Bump and Gamburd allows to do that very easily. We present a generalized version here.
\begin{lemma}[(see also \cite{BumpGamburd})]
\label{lemma.1}
Let  $\kappa$ and $\lambda$ be partitions,  vectorized to length $K$ and $L$ respectively.
Assume there exists a  partition $\mu$ that can be vectorized to length $K+L$ such that\footnote{The symbol $\cup$ refers here to the concatenation of vectors. As such, $(\kappa + \rho_K)\cup (\lambda + \rho_L) \ne (\lambda + \rho_L) \cup (\kappa + \rho_K)$. Once sorted, they are the same, but the sorting permutation will sometimes  have a different sign. } 
\begin{equation}
\mu + \rho_{K+L} = \sort{(\kappa + \rho_K)\cup (\lambda + \rho_L)} \label{eqn.sort}
\end{equation}
and that $|\DDD| = K+L$. Then,
\begin{equation}
\omega(\kappa,\lambda) \cdot  \syms{\mu} {\DDD } =  \sum_{
\substack{  \AAA,\BBB\text{ s.t.}\\ \AAA\cup\BBB = \DDD \\
|\AAA| = K\\
|\BBB| = L }} 
    \frac{\syms{\kappa } {\AAA} \, \syms{\lambda} {\BBB }}{\Delta(\AAA;\BBB)}
   \label{eqn.mainlemma}
\end{equation}
where 
$\omega(\kappa,\lambda)=\pm 1$ is the sign of the sorting permutation in \eqref{eqn.sort} (which is unique as $\mu$ is weakly decreasing). 

If two entries on the RHS of \eqref{eqn.sort} are equal, no $\mu$ can be found as $\mu$ would have to have an increase somewhere. In that case, Equation~\eqref{eqn.mainlemma} still remains true if its LHS is replaced by 0. 
\end{lemma}
In other words, provided we can get to an expression such as \eqref{eqn.summand}, this lemma, when used right-to-left, lets us very easily sum over set partitions and recover an expression for \eqref{eqn.after2} where the full $\mathcal{S}_{2k}$ symmetry is obvious. 

\begin{proof}[ of Lemma~\ref{lemma.1}]
It is easiest to start with a partition $\mu$, and expand $\syms{\mu}{\DDD}$ according to Equation~\eqref{eqn.schur}.  The exponents that appear in the matrix on the numerator are given by $\mu + \rho_{K+L}$. We can use the Laplace expansion on this numerator, separating both the $K+L$ rows (= variables) and the $K+L$ columns (= exponents) into subsets of size $K$ and $L$.  This gives rise to the partition $\AAA \cup \BBB = \DDD$ for the variables and the partition of Equation~\eqref{eqn.sort} for the exponents. We thus have the numerators of the expansions according to Equation~\eqref{eqn.schur} of the Schur functions on the RHS of Equation~\eqref{eqn.mainlemma}, up to a sign. 

The Vandermonde in the denominator for $\syms{\mu}{\DDD}$ can also be split, this time in three subproducts: the factors that involve either two variables ending in $\AAA$, two variables ending in $\BBB$, and one in each.  These subproducts end up respectively as the denominator in $\syms{\kappa}{\AAA}$, $\syms{\lambda}{\BBB}$, or, for the ``cross terms'' in the $\Delta(\AAA;\BBB)$. 
\end{proof}

This lemma is highly convenient, as we eventually need to take a limit $\DDD \rightarrow (2k)\cdot \Zero$. Since $\syms{\mu}{\DDD}$ is homogeneous of degree $|\mu|$ in the variables of $\DDD$, only one term actually matters once we take the limit $\DDD \rightarrow (2k)\cdot \Zero$ of \eqref{eqn.shifts}, the term where $\mu = \phi$ is the empty partition. It is thus desirable to know which pairs of $\kappa,\lambda$ will produce an empty $\mu$ in \eqref{eqn.sort}. Again, this will admit a very simple combinatorial interpretation.

\begin{lemma}
\label{lemma.visual}
Let $\kappa$ and $\lambda$ be two partitions, of partition length less or equal to $K$ and $L$ respectively. Then, 
\[\phi +  \rho_{K+L}  =  \rho_{K+L} = \sort{(\kappa + \rho_K)\cup (\lambda + \rho_L)} \]
if and only if $\lambda^t$, the transpose partition of $\lambda$, has vector $(L-\kappa_{K},\cdots,L-\kappa_1).$ Graphically, this means the Young diagrams of $\kappa$ and $\lambda^t$ can be assembled into a $K\times L$-rectangle. In this case, we also have $\omega(\kappa,\lambda) = (-1)^{|\lambda|} = (-1)^{k^2-|\kappa|}$. 
\label{lemma.2}
\end{lemma}
We say that $\kappa$ and $\lambda$ are $(K,L)-$complementary if they satisfy this last condition, which is illustrated by Figure~\ref{fig.complementary}.
\begin{figure}
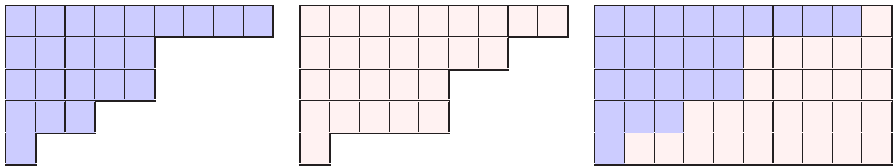

\begin{center}
\begin{tabular}{ccc}
$\young(\backblue\backblue\backblue\backblue\backblue\backblue\backblue\backblue\backblue,\backblue\backblue\backblue\backblue\backblue,\backblue\backblue\backblue\backblue\backblue,\backblue\backblue\backblue,\backblue)$
&
$\young(\backred\backred\backred\backred\backred\backred\backred\backred\backred,\backred\backred\backred\backred\backred\backred\backred,\backred\backred\backred\backred\backred,\backred\backred\backred\backred\backred,\backred)$
&
$\young(\backblue\backblue\backblue\backblue\backblue\backblue\backblue\backblue\backblue\backred,\backblue\backblue\backblue\backblue\backblue\backred\backred\backred\backred\backred,\backblue\backblue\backblue\backblue\backblue\backred\backred\backred\backred\backred,\backblue\backblue\backblue\backred\backred\backred\backred\backred\backred\backred,\backblue\backred\backred\backred\backred\backred\backred\backred\backred\backred)$
\end{tabular}
\end{center}
\caption{The partitions $\kappa = (9,5,5,3,1)$ and $\lambda = (5,$ $4,$ $4,$  $4,$ $4,$ $2,$ $2,$ $1,$ $1)$ are $(5,10)$-complementary, as the Young diagrams of $\kappa$ (represented here in dark blue) and $\lambda^t = (9,7,5,5,1)$ (represented here in light red, before its rotation of $180^\circ$) can be assembled to form a rectangle with 5 rows and 10 columns.}
\label{fig.complementary}
\end{figure}
\begin{proof}[ of Lemma~\ref{lemma.2}]
This proof follows from a classical combinatorial property \cite[p.~3]{Macdonald}:  the partitions $\kappa$ and $\lambda$ are $(K,L)-$complementary if and only if together the sequences $(\kappa_i+K-i: 1 \le i \le K)$ and $(\lambda_i+L-i: 1 \le i \le L)$ are a permutation of the $K+L$ integers between 0 and $K+L-1$.

For, the statement about the sign, we first observe that it is true when $\kappa = \left(L^K\right)$ is the full rectangle and the partition $\lambda$ is empty. Indeed, in that case no sorting is needed. Moreover, if it is true for a pair $(\kappa,\lambda)$, consider a pair of partitions $(\kappa^{-},\lambda^+)$, where $\kappa^-$ is obtained by removing a box from $\kappa$ and $\lambda^+$ by adding the corresponding box to $\lambda$, in such a way that $\kappa^{-},\lambda^+$ are also $(K,L)-$complementary. In terms of the sequences described above, it is easy to check that this transfer of a box merely swaps two of the entries between the sequences. Therefore, it alters the sign by 1. By induction on the size of $\lambda$, we are done. 
\end{proof}

We will only use complementarity in the case of $K=L$, \textit{i.e.}~ a square. In this case, we write 
$\complementary{k}{\lambda}$ for the $(k,k)-$complementary to $\lambda$.

\section{Alternate Ending to the CFKRS Recipe}
\label{sec.alternate}
It will be in our interest to use a slightly different expansion than in Equation~\eqref{eqn.summand}: we will instead expand $\tilde{f}(\AAA;\BBB)$ itself in a symmetric function basis. In fact, in Section~\ref{sec.explicitexpansion}, we will obtain the following proposition, which can be understood as a Taylor expansion for $\tilde{f}$, symmetric in two disjoint sets of variables. 

\begin{prop}
Starting from Equation~\eqref{eqn.after2}, the diagonal term selection in the CFKRS recipe leads to Equation~\eqref{eqn.after3}, with $\tilde{f}(\AAA;\BBB)$ admitting the expansion
\begin{equation}
\label{eqn.expansion}
 \tilde{f}(\AAA;\BBB) =  \sum_{\kappa,\lambda} d_{\kappa\lambda}(-1)^{|\lambda|} \syms{\kappa}{\AAA} \syms{\lambda}{\BBB} =   \sum_{\kappa,\lambda} d_{\kappa\lambda}\syms{\kappa}{\AAA} \syms{\lambda}{\negBBB} =\exp \left( \sum_{\mu,\nu} W_{\mu\nu} \symp{\mu}{\AAA}\symp{\nu}{\negBBB}\right),
\end{equation} 
where the sums run over partitions $\kappa,\lambda$ or $\mu,\nu$ and  
\begin{equation}
W_{\mu\nu} := \sum_{r \ge 1} \VVV{r}{\mu\nu}   \coeff{r}{|\mu|+|\nu|},
\label{eqn.Wmunu}
\end{equation}
with the $\VVV{r}{\mu\nu}$ in $\mathbb{Q}[k]$ (defined during the proof, in Section~\ref{sec.explicitexpansion}) and the $\coeff{r}{n}$ Taylor coefficients of $\primezeta(s)$ at $r$ (defined in Section~\ref{sec.primezeta}). 

Moreover, the $d_{\kappa\lambda}$ are given by polynomials with rational coefficients in the $W_{\mu\nu}$, with total weight in the $\mu$ (resp.~$\nu$) equal to the weight of $\kappa$ (resp.~$\lambda$).
\label{prop.expansion}
\end{prop}
The proof of this proposition is on page~\pageref{proof.prop.expansion}.

\begin{remark*} The constant term of $\tilde{f}(\AAA;\BBB)$, \textit{i.e.}~$d_{\phi\phi} = \exp\left(W_{\phi\phi}\right)$ is well-known, and equals the arithmetic factor present in the leading order conjecture of Keating-Snaith. We present this derivation in Section~\ref{sec.leading}.
\end{remark*}

This last proposition is the missing piece to prove Theorem~\ref{thm.main}.

\begin{proof}[ of Theorem~\ref{thm.main}]
\label{proof.thm.main}
After its three steps, the CFKRS recipe tells us to evaluate
\begin{equation}
\tilde{M}_k(T) := \lim_{\DDD \rightarrow (2k)\cdot\Zero}
\int_0^T   \sum_{\substack{\AAA,\BBB \text{ s.t.}\\ \AAA \cup \BBB = \DDD \\|\AAA|=|\BBB|=k}} \left( \frac{t}{2\pi}\right)^{-\symp{1}{\BBB}}
\frac{\tilde{f}(\AAA;\BBB)}{\Delta(\AAA;\BBB)}\ud t .
\end{equation}
We can expand $\tilde{f}$, thanks to Proposition~\ref{prop.expansion} which takes care of performing the diagonal selection, and gives a very explicit expansion for $\tilde{f}$:
\begin{equation*}
\tilde{M}_k(T)= \lim_{\DDD \rightarrow (2k)\cdot\Zero}
\int_0^T   \sum_{\substack{\AAA,\BBB \text{ s.t.}\\ \AAA \cup \BBB = \DDD \\|\AAA|=|\BBB|=k}} \frac{\sum_{u \ge 0} \frac{(-1)^u}{u!}\left(\log \left( \frac{t}{2\pi}\right)\right)^u \left(\symp{1}{\BBB}\right)^u
\sum_{\kappa,\lambda} d_{\kappa\lambda} (-1)^{|\lambda|} \syms{\kappa}{\AAA} \syms{\lambda}{\BBB}} {\Delta(\AAA;\BBB)}\ud t.
\end{equation*}
We are now at the crux of the argument in this paper: via iterated applications of Pieri's rule (see Section~\ref{sec.pieri}), it will be possible and actually easy to express $\syms{\lambda}{\BBB} (\symp{1}{\BBB})^u$ as a linear combination in the basis of Schur functions in the alphabet $\BBB$, \textit{i.e.}~$\{\syms{\nu}{\BBB}\}$. However, thanks to Lemma~\ref{lemma.1} we know that the outcome of summing any pair $\syms{\kappa}{\AAA}\syms{\nu}{\BBB}$ over alphabet partitions $\DDD = \AAA \cup \BBB$ is actually pretty simple: either 0 or a signed $\syms{\mu}{\DDD}$.  Actually, thanks to Lemma~\ref{lemma.visual}, we know even more: we know which pairs $\syms{\kappa}{\AAA}\syms{\nu}{\BBB}$ will produce a $\pm \syms{\mu}{\DDD}$, with $\mu = \phi$, which is the only $\mu$  to really matter once we take the limit $\DDD \rightarrow (2k)\cdot \Zero$: we need 
$\nu = \complementary{k}{\kappa}$. This can only occur if $|\nu| = k^2 - |\kappa|$, which in turn implies that $u = k^2 - |\kappa|-|\lambda|$. 

We put all those ideas together, and get:
\begin{equation}
\tilde{M}_k(T) =  \int_0^T \sum_{\kappa,\lambda} d_{\kappa\lambda} \frac{(-1)^{k^2 - |\kappa|-{|\lambda|}+{|\lambda|}}{\omega(\kappa,\complementary{k}{\kappa})}}{(k^2 - |\kappa|-|\lambda|)!}\dim(\lambda, \complementary{k}{\kappa}) \left(\log \left(\frac{t}{2\pi}\right) \right)^{k^2 - |\kappa|-|\lambda|}
\ud t.
\end{equation}
Remark that $(-1)^{k^2 - |\kappa|}{\omega(\kappa,\complementary{k}{\kappa})} =1$ by Lemma~\ref{lemma.visual}.

\end{proof}

\section{Explicit Expansion as Symmetric Functions}
\label{sec.explicitexpansion}
We still need to explicitly compute the transition brought in our expressions by selecting the diagonal terms. We thus need to prove Proposition~\ref{prop.expansion}.

\begin{proof}[of Proposition~\ref{prop.expansion}] 
\label{proof.prop.expansion}
We need to go back to the summands in Expression~\eqref{eqn.after2}, and apply the diagonal term selection to obtain $\tilde{f}(\AAA;\BBB)$.  We will first express $\log(\tilde{f}(\AAA;\BBB))$ in the basis of doubly-symmetric functions $ \symp{\mu}{\AAA}\symp{\nu}{\BBB}$ and prove the results about the $W_{\mu\nu}$.  Statements about the $d_{\kappa\lambda}$ will only come at the very end.

%\prod_\alpha \sum_{n\ge1} \frac{1}{n^{\frac{1}{2} + \alpha + \mathfrak{i}t}} \cdot
%\prod_\beta  \sum_{m\ge1} \frac{1}{m^{\frac{1}{2} - \beta - \mathfrak{i}t}}\right\}

We first need some new notation. Let $\QQQ$ be the set\footnote{It would be interesting to use $\QQQ$ as an alphabet, but it has proved difficult to do that with any advantage.} composed of the inverse of primes: $\QQQ := \{ 1/p: p \text{ prime}\}$. It will be implicit that we have the variable $Q$ run through $\QQQ$. This prevents the overloading of the letter ``p'', that would soon occur otherwise. 

We can now consider the innermost terms of Expression~\eqref{eqn.after2}:
\begin{align*}
\prod_\alpha \sum_{n\ge1} \frac{1}{n^{\frac{1}{2} + \alpha + \mathfrak{i}t}} \cdot
\prod_\beta  \sum_{m\ge1} \frac{1}{m^{\frac{1}{2} - \beta - \mathfrak{i}t}} 
&= \prod_Q \left( \sum_{u \ge 0} \symh{u}{Q^\AAA}  Q^{u/2+\mathfrak{i}tu} \sum_{v \ge 0} \symh{v}{Q^\negBBB}  Q^{u/2-\mathfrak{i}tu} \right)\\
&\Rightarrow \prod_Q \left( \sum_{u \ge 0} \symh{u}{Q^\AAA}  \symh{u}{Q^\negBBB}  Q^{u} \right).
\end{align*}
This was the diagonal selection step: for each $Q$, the powers of $Q$ have to match, up to the slight shifts introduced by $\AAA$ and $\BBB$. We continue from that last expression, and use first a transition rule described in Section~\ref{sec.symmetric}. 
\begin{align}
\prod_Q \left( \sum_{u \ge 0} \symh{u}{Q^\AAA}  \symh{u}{Q^\negBBB}  Q^{u} \right)  
&= \prod_Q \left( \sum_{\substack{\kappa,\lambda\\|\kappa| = |\lambda|}} \frac{1}{z_\kappa z_\lambda} \symp{\kappa}{Q^\AAA}  \symp{\lambda}{Q^\negBBB}  Q^{|\kappa|} \right)\nonumber\\
&=  \prod_Q \exp \left( \sum_{\substack{\kappa,\lambda\\|\kappa| = |\lambda| \ge 1}} f_{\kappa \lambda} \symp{\kappa}{Q^\AAA}  \symp{\lambda}{Q^\negBBB}  Q^{|\kappa|} \right)\label{eqn.fkl}\\
&=\exp \left(\sum_{\substack{\kappa,\lambda\\|\kappa| = |\lambda| \ge 1}}  f_{\kappa \lambda} \sum_Q  \symp{1}{Q^{\{|\kappa|\}\calplus \phi_\kappa(\AAA)\calplus \phi_\lambda(\negBBB)}} \right) \nonumber\\
&=  \exp \left(\sum_{\substack{\kappa,\lambda\\|\kappa| = |\lambda| \ge 1}}  f_{\kappa \lambda} \primezeta\Bigl(\{|\kappa|\}\calplus \phi_\kappa(\AAA)\calplus \phi_\lambda(\negBBB)\Bigr) \right)\label{eqn.ready}
\end{align}
The $f_{\kappa\lambda}$ are defined and computed by using the power series for the logarithm on the RHS of the first line, which is valid as $z_\phi=1$. They do not depend on $k$. More information about them is available in Section~\ref{sec.fkl}.

It is clear from here that we should use the power series expansions for the prime zeta function given in Section~\ref{sec.primezeta}. The value of $|\kappa|$ will determine where we base our expansion. The value of $f_{(1)(1)}$, and $\phi_{(1)}(\AAA) \calplus \phi_{(1)}(\negBBB) = \AAA \calplus \negBBB$ are needed to derive the next lines, together with the expansions in Section~\ref{sec.primezeta}: 
\begin{align}
\label{eqn.ak}\nonumber
\eqref{eqn.ready} &= \exp \left(\sum_{\substack{\kappa,\lambda\\|\kappa| = |\lambda| \ge 1}}  f_{\kappa \lambda} \primezeta\Bigl(\{|\kappa|\}\calplus \phi_\kappa(\AAA)\calplus \phi_\lambda(\negBBB)\Bigr) \right)
\\\nonumber
&=\exp \left(\sum_{\substack{\kappa,\lambda\\|\kappa| = |\lambda| \ge 1}}  f_{\kappa \lambda} \sum_{u\ge 0} \coeff{|\kappa|}{u} \symp{u}{{\phi_\kappa(\AAA)\calplus \phi_\lambda(\negBBB)} }  - \sum_{\gamma\in\AAA\calplus\negBBB}\log(\gamma) \right)\\
\nonumber
&= \frac{1}{\Delta(\AAA;\BBB)}  \exp\left( \sum_{\substack{\kappa,\lambda\\|\kappa| = |\lambda| \ge 1}}  f_{\kappa \lambda}  \sum_{m,n=0}^\infty \coeff{|\kappa|}{m+n} \binom{m+n}{m,n}\symp{m}{\phi_\kappa(\AAA)}\symp{n}{\phi_\lambda(\negBBB)} \right)\\
\nonumber
&=  \frac{ \exp\left({\substack{ \displaystyle{\sum_{\substack{\kappa,\lambda\\|\kappa| = |\lambda| \ge 1}}  f_{\kappa \lambda}  \sum_{{\mu,\nu}}  \coeff{|\kappa|}{|\mu|+|\nu|} \binom{|\mu|+|\nu|}{\mu\cup \nu}}\times\hfill\\
\displaystyle{ \quad\quad\quad\quad\quad\quad\quad\quad\symm{\mu}{\kappa}   \symm{\nu}{\lambda} 
 |\AAA|^{l(\kappa)-l(\mu)} |\BBB|^{l(\lambda)-l(\nu)}
 \symp{\mu}{\AAA}  \symp{\nu}{\negBBB}}}}\right)}{\Delta(\AAA;\BBB)} \\
&=  \frac{1}{\Delta(\AAA;\BBB)}  \exp\left(
   \sum_{\mu,\nu} \sum_{r \ge 1} \VVV{r}{\mu\nu}  
  \coeff{r}{|\mu|+|\nu|} 
  \symp{\mu}{\AAA}
   \symp{\nu}{\negBBB}\right)\nonumber 
\end{align}
if we use the notational shortcut $\binom{|\lambda|}{\lambda} = \binom{|\lambda|}{\lambda_1,\cdots,\lambda_{l(\lambda)}}$, remember we had assumed $|\AAA| = |\BBB|=k$, 
and define 
\begin{equation}\VVV{r}{\mu\nu} :=\binom{|\mu|+|\nu|}{\mu\cup \nu} \sum_{\kappa,\lambda\vdash r}  f_{\kappa \lambda}   \symm{\mu}{\kappa}  \symm{\nu}{\lambda} k^{l(\kappa)+l(\lambda)-l(\mu)-l(\nu)}.
\label{eqn.Vrmunu}
\end{equation}
for any pair of partitions $\mu, \nu$ and $r$ a positive integer. For fixed $r,\mu,\nu$, $V$ is a polynomial in $k$, as we need $l(\kappa) \ge l(\mu), l(\lambda) \ge l(\nu)$ for the $\symm{\mu}{\kappa}$ and $\symm{\nu}{\lambda}$ to be nonzero. The exponent of $k$ can therefore never be negative.

We also deduce from this formula that the degree of $\VVV{r}{\mu\nu}$ in $k$ is bounded by $2r-l(\mu)-l(\nu)$.

The $d_{\kappa\lambda}$ are obtained by expanding the exponential and using the transition matrix between Schur and power polynomials (Equation~\eqref{eqn.transitionpower}). They are therefore algebraic expressions in the $W_{\mu\nu}$, with rational coefficients and with the total weight of the partitions involved in the $W$ less or equal to $|\kappa|+|\lambda|$. 
\end{proof}
\section{Proof of Theorem~\ref{thm.dim}}
\label{sec.thm3}
We wish to prove in this section Theorem~\ref{thm.dim}. This is of independent interest.

Olshanski, Regev and Vershik have proved results on the dimensions of skew partitions. Their perspective is a little different from ours, but we can still use their results to prove the theorem.

\subsection{Frobenius-Schur functions}
\label{sec.FS}
Frobenius-Schur functions are used to compute the dimensions $\dim(\lambda,\mu)$. We need to introduce the basics from the beautiful theory of \cite{FS}.

\subsubsection{Supersymmetric functions}
Define \emph{supersymmetric power polynomials} in two alphabets as 
\begin{align*}
\symp{r}{\AAA;\BBB} &:= \sum \alpha^r + (-1)^{r-1} \sum \beta^r,\\
\symp{\lambda}{\AAA;\BBB} &:= \prod_i \symp{\lambda_i}{\AAA;\BBB}
\end{align*}

We keep the same transition matrices for supersymmetric polynomials as for standard symmetric polynomials, so this also defines the \emph{supersymmetric Schur functions} $\syms{\lambda}{\AAA;\BBB}$ (via the rule given in Equation~\eqref{eqn.transitionschur}).

\subsubsection{Shifted Frobenius coordinates}
Given a partition $\lambda$, we define the two finite sequences of nonnegative integers $p_i := \lambda_i -i$ and $q_i := \lambda_i^t-i$ for $1\le i\le d$ (it turns out indeed that they both have length $d = d(\lambda)$, given as the number of entries on the main diagonal of the Ferrers diagram of $\lambda$. This $d$ is called the \emph{Frobenius dimension} of $\lambda$). This pair of finite sequences is the \emph{Frobenius coordinates} of $\lambda$ noted $(p_i|q_i)$.
For instance, the partition $(9,5,5,3,1)$ of Figure~\ref{fig.complementary} has Frobenius coordinates $(8,3,2|4,2,1)$. The partition $(9,7,5,5,1)$, in light red in Figure~\ref{fig.complementary}, has Frobenius coordinates $(8,5,2,1|4,2,1,0)$. The partition $(1)$ would have Frobenius coordinates $(0|0)$ and the empty partition's Frobenius coordinates are $(|)$. 

If $\lambda $ has Frobenius coordinates $(p_1,\cdots, p_d|q_1,\cdots,q_d)$, define $x(\lambda) := (p_i+1/2: 1 \le i \le d)$ and $y(\lambda) := (q_i+1/2: 1 \le i \le d)$. The pair $(x(\lambda),y(\lambda))$ is then the \emph{shifted Frobenius coordinates}. Note that $\sum_{x \in x(\lambda)} x + \sum_{y \in y(\lambda)} y = |\lambda|$. 

\subsubsection{Definition of Frobenius-Schur functions}  Let $\mu, \nu$ be partitions, with $m:= |\mu|, n := |\nu|$. The purpose of the Frobenius-Schur functions is precisely to compute $\dim(\lambda,\mu)$. In fact, this problem serves as their definition, by interpolation of the following formula:
\begin{equation}
\frac{\dim(\mu,\nu)}{\dim(\phi,\nu)} = \frac{F\syms{\mu}{x(\nu);y(\nu)}}{n (n-1)\cdots (n-m+1)}.
\end{equation}

We caution the reader that $F\symssymb{\lambda}$ is in general not homogeneous. However it does agree with $\symssymb{\lambda}$ for its terms of top total degree.

\subsubsection{Giambelli formula} Given two integers $a$, $l$, define the \emph{hook partition} of \emph{arm length} $a$ and \emph{leg length} $l$ as the partition with Frobenius coordinates $(a|l)$. By abuse of notation, we also denote the partition itself by $(a|l)$ then.

The Giambelli formula, proved for Frobenius-Schur functions in \cite[an equation between (2.4) and (2.5)]{FS}, is a formula that gives an algebraic expressions for Frobenius-Schur functions in terms of Frobenius-Schur functions of hook partitions.  Let $\lambda$ be a partition with Frobenius dimension $d=d(\lambda)$ and Frobenius  coordinates $(p_1, \cdots, p_d|q_1, \cdots , q_d)$. Then, 
\begin{equation}
\label{eqn.giambelli}
F\syms{\lambda}{\AAA;\BBB} =\det \left[   F\syms{{(p_i|q_j)}}{\AAA;\BBB}  \right]_{i,j=1}^{d}.
\end{equation}

We still need to know how to compute Frobenius-Schur functions of hook partitions.
\subsubsection{Frobenius-Schur functions for hook partitions}
We simply quote \cite[Proposition 2.4]{FS}:
\begin{equation}
\label{eqn.forhook}
F\syms{(p|q)}{\AAA;\BBB} = \sum_{p'=0}^p \sum_{q'=0}^q c_{pp'}c_{qq'} \syms{(p'|q')}{\AAA;\BBB},
\end{equation}
where
\begin{equation}
c_{pp'} = 
\begin{cases}
(-1)^{p-p'} \syme{p-p'}{\frac{1}{2},\frac{3}{2},\cdots,\frac{2p-1}{2}},&p' \le p;\\
0, &p'>p.
\end{cases}
\end{equation}
The notation $\syme{r}{\AAA}$ is a notational shortcut for  $\syms{\left(1^r\right)}{\AAA}$ (these functions are also part of the  \emph{elementary symmetric polynomials} family). Note that $c_{pp}=1$.
\subsection{The proof itself}
We are now ready to prove Theorem~\ref{thm.dim}.
\begin{proof}[ of Theorem~\ref{thm.dim}]
\label{proof.thm.dim}
We know from the results on Frobenius-Schur functions that:
\begin{equation}
\label{eqn.double}
\frac{\dim(\kappa,\complementary{k}{\lambda})}{\dim(\phi,\complementary{k}{\lambda})} = \frac{F\syms{\kappa}{x(\complementary{k}{\lambda});y(\complementary{k}{\lambda})}}{(k^2-|\lambda|)\cdots(k^2-|\kappa|-|\lambda|+1)}.
\end{equation}
By symmetry considerations, we also have
\begin{equation}
\dim(\phi,\complementary{k}{\lambda}) = \dim(\lambda,\complementary{k}{\phi}) = \dim(\lambda,(k^k)),
\end{equation}
as the truncation of the Young lattice to partitions fitting inside a $k\times k$ rectangle exhibits additional symmetry.

We use Frobenius-Schur functions again, and obtain
\begin{equation}
\label{eqn.single}
\frac{\dim(\lambda,(k^k))}{\dim(\phi,(k^k))} = \frac{F\syms{\lambda}
{x((k^k));y((k^k))}
}{k^2\cdot(k^2-1) \cdots (k^2-|\lambda|+1)}.
\end{equation}
Therefore, we obtain
\begin{equation}
\dim(\kappa,\complementary{k}{\lambda}) = \frac{\dim(\phi,(k^k)) F\syms{\kappa}{x(\complementary{k}{\lambda});y(\complementary{k}{\lambda})}F\syms{\lambda}
{x(\complementary{k}{\phi});y(\complementary{k}{\phi})}
}{k^2\cdot(k^2-1)\cdots(k^2-|\kappa|-|\lambda|+1)} .
\end{equation}
We will show in the proof of Theorem~\ref{thm.dim}, around Equation~\eqref{eqn.hookformula}, that $ \dim(\phi,(k^k)) = g_k$. Therefore, we are done if we can show the general statement that 
\begin{equation*}
F\syms{\mu}{x(\complementary{k}{\nu});y(\complementary{k}{\nu})}
\end{equation*}
is a polynomial in $k$ of degree at most $2|\mu|$. By the Giambelli formula, it is sufficient to prove this for $\mu$ a hook partition. By Equation~\eqref{eqn.forhook}, the problem is further reduced to supersymmetric Schur polynomials. Finally, by the transition rule \eqref{eqn.transitionschur} (and multiplicativity of the power polynomials),  we only need to prove polynomiality in $k$ for
\begin{equation*}
\symp{r}{x(\complementary{k}{\nu});y(\complementary{k}{\nu})},
\end{equation*}
together with a bound on the degree of the polynomial in $k$ of $2r$ (under the hypothesis $r \ge 1$). 

It is easy to see that there exist two fixed, finite sets $I_x$ and $I_y$ such that
\begin{align*}
x(\complementary{k}{\nu}) &= I_x \, \Delta  \, x((k^k)) \\
y(\complementary{k}{\nu}) &= I_y \, \Delta  \, y((k^k)) 
\end{align*}
for all $k$ (with $\Delta$ symmetric difference of sets).  Loosely said, up to a few initial fixes, the Frobenius coordinates for $\complementary{k}{\nu}$ and the full $k \times k$ square are the same for each $k$. This implies
\begin{equation*}
\symp{r}{x(\complementary{k}{\nu});y(\complementary{k}{\nu})} = \symp{r}{x((k^k));y((k^k))} - \symp{r}{I_x;I_y},
\end{equation*}
which is valid for all $k$. It becomes clear, since the last term in the RHS is a constant, that we only need to concern ourselves with the first term in the RHS. This term is much simpler, and can actually be explicitly computed:
\begin{align*}
\symp{r}{x((k^k));y((k^k))} &= \symp{r}{k-\frac12,\cdots,\frac12;k-\frac12,\cdots,\frac12}\\
&= \begin{cases}
2 \sum_{i=0}^{k-1} (i+\frac12)^r, & r\text{ is odd;}\\
0, & r\text{ is even.}
\end{cases}
\end{align*}
Since $\sum_{i=0}^{k-1} (i+\frac12)^r$ is a polynomial in $k$ (by Faulhaber's formula) of degree $r+1 \le 2 r$ (since we have $r \ge 1$),  we are done.
\end{proof}

\begin{remark*} This proof unfortunately breaks some of the symmetry of the statement of the theorem, giving a different r\^ole to $\kappa$ and $\lambda$. This is due to the current state of the theory on Frobenius-Schur functions. Careful observation shows that we have implicitly proved along the way that
\begin{align*}
F\syms{\kappa}{x(\complementary{k}{\lambda});y(\complementary{k}{\lambda})}F\syms{\lambda}
{x((k^k));y((k^k))} = \\
F\syms{\lambda}{x(\complementary{k}{\kappa});y(\complementary{k}{\kappa})}F\syms{\kappa}
{x((k^k));y((k^k))}.
\end{align*}
In fact, one should expect a factorization for the final formula for  $\dim(\kappa,\complementary{k}{\lambda})$, that would give a symmetric r\^ole to $\kappa$ and $\lambda$.
\end{remark*}

\section{Corollaries of the Main Theorem}
\label{sec.corollaries}
\subsection{Leading term}
\label{sec.leading}
We had actually left out one part of the proof of Theorem~\ref{thm.main}, the proof that $ g_k =  \dim(\phi,(k^k))$. Indeed, we want our computations to be completely independent of previous work of Keating and Snaith for instance.  In particular, we want to compute the leading coefficient of $P_k$ using only Theorem~\ref{thm.main}. We feel that interesting ideas come from this proof.
\begin{prop}[(see \cite{CFKRS2})]
\label{prop.leading}
The leading term $c_0(k)$ of $P_k$ is given by 
\begin{equation}
c_0(k) = \frac{ a_k g_k }{(k^2)!}
\end{equation}
with $g_k$ also given by $ \frac{k^2!}{\prod_{i,j=0}^{k-1} (i+j+1)}  .$
\end{prop}
\begin{proof}
For the leading term of $P_k$, the main sum of Theorem~\ref{thm.main} is reduced to just  one term, corresponding to $\kappa=\lambda= \phi$. Exploiting that $\complementary{k}{\phi}$ is the full $k\times k$ square $(k^k)$, we have
\begin{equation}
c_0(k) =  d_{\phi\phi} \frac{\dim(\phi,(k^k))}{(k^2)!}.
\end{equation}
We will show that $\dim(\phi,(k^k))=g_k$  and that $d_{\phi\phi} = a_k$.

\textbf{The combinatorial/RMT factor $g_k$. }
The new definition given for $g_k$ in the statement of the proposition is trivially equivalent to the original definition of Equation~\eqref{eqn.gk}. 

The proof that  $\dim(\phi, (k^k)) = \frac{k^2!}{\prod_{i,j=0}^{k-1} (i+j+1)} $ will amount to the hook length formula \cite{FRT}, which gives the $\dim(\phi,\lambda)$ for a partition $\lambda$: 
\begin{equation}
\dim(\phi,\lambda) = \frac{|\lambda|!}{H(\lambda)},
\label{eqn.hookformula}
\end{equation}
with $H(\lambda) = \prod_{\square \in \lambda} h(\square)$ and $h(\square)$ the length of the hook based at the box $\square$ of $\lambda$. 

For a $k\times k$ square partition, the product of hook lengths is given by $\prod_{i,j=0}^{k-1} (i+j+1)$, so we are done.  

Equivalently, we are computing here the number of Young tableaux of shape $k \times k$. This identity between $g_k$ and the number of such tableaux was already known (see OEIS A039622 \cite{A039622}), but obscured by the connection with random matrices. It is however more natural, as it shows the integrality property of the $g_k$, and does not require an extraneous limit in $N$, as in the definition of Equation~\eqref{eqn.gk}.

\textbf{The arithmetic factor $a_k$. } 
We want to show
\begin{equation}
\label{prop.ak}
\prod_p \left(1-\frac{1}{p}\right)^{k^2} {}_2F_1\left(k,k;1;\frac{1}{p}\right) =: a_k  = \exp\left(W_{\phi\phi}  \right) = d_{\phi\phi} 
\end{equation}
The first equality is a reminder of the definition of $a_k$. The last equality is a consequence of Equation~\eqref{eqn.expansion}. We thus only need to prove the remaining equality, which amounts to $\log a_k = W_{\phi\phi} $. 

We start with Equation~\eqref{eqn.Wmunu}, and obtain
\begin{align}
 W_{\phi\phi} \nonumber
&= \sum_{r \ge 1} \VVV{r}{\phi\phi}   \coeff{r}{0}\nonumber\\
& =\sum_{r \ge 1} \left(\sum_{\kappa,\lambda\vdash r}  f_{\kappa \lambda}   k^{l(\kappa)+l(\lambda)}\right) \coeff{r}{0}  \quad\quad  \text{(by Eqn.~\eqref{eqn.Vrmunu})}\nonumber\\
&= f_{(1)(1)} \coeff{1}{0} k^2 +  \sum_{r \ge 2} \sum_{\kappa,\lambda\vdash r}  f_{\kappa \lambda} \coeff{r}{0} k^{l(\kappa)+l(\lambda)}\nonumber\\
\label{sec.step1}
&=f_{(1)(1)}  k^2  \sum_{n\ge 2} \frac{\mu(n)}{n}  \log \zeta(n) +  \sum_{r \ge 2} \sum_Q \ \left(Q^r \sum_{\kappa,\lambda\vdash r }  f_{\kappa \lambda}   \symp{\kappa}{k\cdot\One}  \symp{\lambda}{k\cdot\One}\right) \\
&=- k^2  \sum_Q \sum_{n\ge 2} \frac{\mu(n)}{n}  \log (1-Q^n) +   \sum_Q \sum_{r \ge 2} \left(Q^r \sum_{\kappa,\lambda\vdash r }  f_{\kappa \lambda}   \symp{\kappa}{k\cdot\One}  \symp{\lambda}{k\cdot\One}\right) \label{sec.step2}\\
&=   \sum_Q \left\{ \left(- k^2 \sum_{n\ge 2} \frac{Q^{n}}{n}\right) -k^2 Q+   \sum_{r \ge 1} \left(Q^r \sum_{\kappa,\lambda\vdash r }  f_{\kappa \lambda}   \symp{\kappa}{k\cdot\One}  \symp{\lambda}{k\cdot\One}\right) \right\} \label{sec.step3}\\
&= \sum_Q  \left\{  k^2  \left(-\sum_{n\ge 1} \frac{Q^{n}}{n}\right) +   \log \left(\sum_{u \ge 0 }\symh{u}{k\cdot\One}\symh{u}{k\cdot\One}Q^u  \right) \right\}. \label{sec.step4}
\end{align}

Expression~\eqref{sec.step1} follows from the definition of the $\coeff{r}{n}$ and the definition of each  $\symp{\kappa_i}{k\cdot \One}$. Expression~\eqref{sec.step2} rearranges absolutely convergent sums. 
 Expression~\eqref{sec.step3} follows from M\"obius inversion, and  adds/removes the term $-k^2 Q$. 
We go back through  the definition of the $f_{\kappa\lambda}$ to obtain Expression~\eqref{sec.step4}.

We are now done, since $\symh{u}{k\cdot\One} = \binom{u+k-1}{k-1}$.
\end{proof}

\begin{remark*} Conceptually, this proof is quite simple: we want the constant coefficient in $\AAA$ and $\BBB$, which means that we will replace $\AAA$ (resp.~$\BBB$) by the alphabet $|\AAA| \cdot \Zero$, satisfying the properties $\symp{\kappa}{|\AAA| \cdot \Zero} =0$ when $|\kappa| >0$,  $\symp{\phi}{|\AAA| \cdot \Zero} =1$ and $\symp{0}{|\AAA| \cdot \Zero} =|\AAA|$ (these values have already been used to obtain  Equation~\eqref{eqn.Vrmunu}). We simply have to be careful (just as CFKRS), since some of the infinite products might be divergent if handled inappropriately. In the proof, each of the sums over $Q$ converges absolutely, as none of the summands have a remaining $Q^0$ or $Q^1$ term (in Expressions~\eqref{sec.step3} and \eqref{sec.step4}, there is cancellation in the $Q^1$ term and no $Q^0$ term).
\end{remark*}

\subsection{Uniform convergence of coefficients}
\label{sec.uniform}
In \cite{Rub1}, Rubinstein and Hiary provide uniform asymptotics for roughly the first $k$ coefficients of $P_k$. This relies on intricate and separate bounds for the arithmetic and combinatorial components.

Our method does not bring in significant information for the arithmetic factor.

 On the other hand, the direct combinatorial interpretation we provide in this paper might be useful to obtain bounds on the combinatorial factor. In particular, Equation~\eqref{eqn.double}, which reduces up to a polynomial the computation of $\dim(\kappa,\complementary{k}{\lambda}) $ to the computation of $\dim(\phi,\complementary{k}{\lambda})$,  is reminescent of Lemma 3.3 of \cite{Rub1}, which reduces up to a polynomial the computation of their $p_k(\alpha_1,\cdots,\alpha_k;\alpha_{k+1},\cdots,\alpha_{2k})$ to the computation of their $p_k(\lambda;0)$. They then translate the computation of the $p_k(\lambda;0)$ into computation of the polynomials $N_k^0(\lambda)$, via a recurrence formula. On our side, we do not need such a recurrence formula, as the theory of Frobenius-Schur functions provide more direct ways to compute the  $\dim(\phi,\complementary{k}{\lambda})$. Indeed, the way we presented in Equation~\eqref{eqn.giambelli} only requires computing a $d\times d$ determinant, where $d = d(\lambda)$ is the Frobenius dimension (a quantity that grows at most like $\sqrt{|\lambda|}$). 
 
This is significant for the results of \cite{Rub1}, as the main theorem there only concerns the roughly first $k$ coefficients:  for a fixed $\alpha <1$, the asymptotics in $k$ for $c_r(k)$ are uniform in the range $0 \le r <k^\alpha$. These asymptotics depend on three bounds in \cite{Rub1}: Equations~(50), (51) and (52). Equation (50) is the only one with this type of restriction on the range of $r$. The other two involve the arithmetic factor.  It  our point of view, it is easily understood why such a restriction might appear: when $|\kappa|+|\lambda| < k $, we know that $\dim(\kappa,\complementary{k}{\lambda})$ will clearly be nonzero, for instance, as $\kappa \subset \complementary{k}{\lambda}$ then. When $|\kappa|+|\lambda| \sim k^\alpha $ for $\alpha >1$, the situation becomes more tricky as the partitions $\kappa$ and $\lambda^t$ fitted inside a $k\times k$ square might start overlapping. 

\section{(Pre)computation of the Constants}
\label{sec.precomputation}
The structure of Theorem~\ref{thm.main} is such that it relies on several families of constants: the $\coeff{r}{n}$, $f_{\kappa\lambda}$, $\VVV{r}{\mu\nu}$,  $W_{\mu\nu}$, $d_{\mu \nu}$, and $\dim(\kappa,\lambda)$. 
We explain how to compute the first five families in this section, as we have already explained in Section~\ref{sec.thm3} how to compute $\dim(\kappa,\lambda)$.

\subsection{Precomputation of  the $\coeff{r}{n}$}
\note{Discuss computations with prime zeta, theoretical results on decrease of the coefficients, and give table}
\label{sec.cnr}
These were defined in \eqref{eqn.coeffrn}, page~\pageref{eqn.coeffrn}. 

Their computation only relies on Equation~\eqref{eqn.moebius}, page~\pageref{eqn.moebius}, which is well-known and old. The computation itself involves no new technique, but does not seem to appear explicitly in the literature. 

By taking successive formal derivations of Equation~\eqref{eqn.coeffrn} using the quotient rule, we can compute $\coeff{r}{N}$ as a series over the squarefree integers $q$ of a (fixed) rational function in the $\zeta^{(n)}(qr)$, with $n \le N$ (if $N$ equals 0, we actually have to take $\log(\zeta(qr))$ instead). If $r=1$, when $q=1$, the multiple $1\cdot 1$ has to be ``corrected'' according to Equation~\eqref{eqn.coeffr1} to remove the singularity.

The main point here is that $\zeta^{(n)}(qr)$ will be easily computable by Euler-Maclaurin summation, and that the terms in all the series involved are decreasing exponentially fast to 0, so their truncations provide very good approximations. Additional tricks can also used to improve convergence, as explained in \cite{Cohen}.
\subsection{Precomputation of the $f_{\kappa\lambda}$}
\label{sec.fkl}
It is clear from their definition in Equation~\eqref{eqn.fkl} that the $f_{\kappa\lambda}$ present the  symmetry $f_{\kappa\lambda} = f_{\lambda\kappa}$, with the additional restriction that $f_{\kappa\lambda} =0$ if $|\kappa| \ne |\lambda|$. 

We present some actual values in Table~\ref{table.fkl}. These were computed using \textsf{sage} \cite{sage} to perform computations in the lazy power series ring over the tensor product of the symmetric polynomials algebra $\Lambda$ with itself. \textsf{Sage} is helpful as it includes all the information necessary to transition between bases of $\Lambda$ and therefore also $\Lambda(\AAA) \otimes \Lambda(\BBB)$.
\begin{table}[tph]
\[
\begin{array}{c||c|cc|ccc}
\kappa,\lambda&{\tiny \yng(1)}&{\tiny \yng(2)}&{\tiny \yng(1,1)}&{\tiny \yng(1,1,1)}&{\tiny \yng(2,1)}&{\tiny \yng(3)}\\
\hline
\hline
{\tiny \yng(1)}&1&&&&&\\[10pt]
\hline
{\tiny \yng(2)}&&\frac14&\frac14&&\\[10pt]
{\tiny \yng(1,1)}&&\frac14&-\frac14&&\\[10pt]
\hline
{\tiny \yng(1,1,1)}&&& &\frac19&-\frac16&\frac{1}{18}\\[10pt]
{\tiny \yng(2,1)}&&&&-\frac16&&\frac16\\[10pt]
{\tiny \yng(3)}&&&&\frac{1}{18}&\frac16&\frac19
\end{array}
\]
\caption{The values of $f_{\kappa\lambda}$, with empty entries of value 0, for small partition sizes.\label{table.fkl}}
\end{table}

\begin{table}[tph]
\[
\begin{array}{c||ccccc}
&{\tiny{\yng(4)}}&{\tiny{\yng(3,1)}}&{\tiny{\yng(2,2)}}&{\tiny{\yng(2,1,1)}}&{\tiny{\yng(1,1,1,1)}}\\
\hline\hline
{\tiny{\yng(4)}}&\frac{1}{16}&\frac{1}{12}&\frac{1}{32}&\frac{1}{16}&\frac{1}{96}\\[10pt]
{\tiny{\yng(3,1)}}&\frac{1}{12}&0&\frac{1}{24}&-\frac{1}{12}&-\frac{1}{24}\\[10pt]
{\tiny{\yng(2,2)}}&\frac{1}{32}&\frac{1}{24}&-\frac{1}{64}&-\frac{1}{32}&-\frac{5}{192}\\[10pt]
{\tiny{\yng(2,1,1)}}&\frac{1}{16}&-\frac{1}{12}&-\frac{1}{32}&-\frac{1}{16}&\frac{11}{96}\\[10pt]
{\tiny{\yng(1,1,1,1)}}&\frac{1}{96}&-\frac{1}{24}&-\frac{5}{192}&\frac{11}{96}&-\frac{11}{192}
\end{array}
\]
\caption{The values of $f_{\kappa\lambda}$, for partitions of  size 4.\label{table.fkl4}}
\end{table}

\begin{table}[tph]
\[
\begin{array}{c||ccccccc}
&{\tiny{\yng(5)}}&{\tiny{\yng(4,1)}}&{\tiny{\yng(3,2)}}&{\tiny{\yng(3,1,1)}}&{\tiny{\yng(2,2,1)}}&{\tiny{\yng(2,1,1,1)}}&{\tiny{\yng(1,1,1,1,1)}}\\
\hline\hline
{\tiny{\yng(5)}}&\frac{1}{25}&\frac{1}{20}&\frac{1}{30}&\frac{1}{30}&\frac{1}{40}&\frac{1}{60}&\frac{1}{600}\\[10pt]
{\tiny{\yng(4,1)}}&\frac{1}{20}&0&\frac{1}{24}&-\frac{1}{24}&0&-\frac{1}{24}&-\frac{1}{120}\\[10pt]
{\tiny{\yng(3,2)}}&\frac{1}{30}&\frac{1}{24}&0&0&-\frac{1}{48}&-\frac{1}{24}&-\frac{1}{80}\\[10pt]
{\tiny{\yng(3,1,1)}}&\frac{1}{30}&-\frac{1}{24}&0&0&-\frac{1}{16}&\frac{1}{24}&\frac{7}{240}\\[10pt]
{\tiny{\yng(2,2,1)}}&\frac{1}{40}&0&-\frac{1}{48}&-\frac{1}{16}&0&\frac{1}{48}&\frac{3}{80}\\[10pt]
{\tiny{\yng(2,1,1,1)}}&\frac{1}{60}&-\frac{1}{24}&-\frac{1}{24}&\frac{1}{24}&\frac{1}{48}&\frac{1}{12}&-\frac{19}{240}\\[10pt]
{\tiny{\yng(1,1,1,1,1)}}&\frac{1}{600}&-\frac{1}{120}&-\frac{1}{80}&\frac{7}{240}&\frac{3}{80}&-\frac{19}{240}&\frac{19}{600}
\end{array}\]
\caption{The values of $f_{\kappa\lambda}$, for partitions of  size 5.\label{table.fkl5}}
\end{table}

\begin{table}[tph]
\[
\begin{array}{c||ccccccccccc}
&({6}^{})&({5}^{},{1}^{})&({4}^{},{2}^{})&({4}^{},{1}^{2})&({3}^{2})&({3}^{},{2}^{},{1}^{})&({3}^{},{1}^{3})&({2}^{3})&({2}^{2},{1}^{2})&({2}^{},{1}^{4})&({1}^{6})\\
\hline\hline
({6}^{})&\frac{1}{36}&\frac{1}{30}&\frac{1}{48}&\frac{1}{48}&\frac{1}{108}&\frac{1}{36}&\frac{1}{108}&\frac{1}{288}&\frac{1}{96}&\frac{1}{288}&\frac{1}{4320}\\[10pt]
({5}^{},{1}^{})&\frac{1}{30}&0&\frac{1}{40}&-\frac{1}{40}&\frac{1}{90}&0&-\frac{1}{45}&\frac{1}{240}&-\frac{1}{80}&-\frac{1}{80}&-\frac{1}{720}\\[10pt]
({4}^{},{2}^{})&\frac{1}{48}&\frac{1}{40}&0&0&\frac{1}{144}&0&-\frac{1}{72}&-\frac{1}{192}&-\frac{1}{64}&-\frac{1}{64}&-\frac{7}{2880}\\[10pt]
({4}^{},{1}^{2})&\frac{1}{48}&-\frac{1}{40}&0&0&\frac{1}{144}&-\frac{1}{24}&\frac{1}{36}&-\frac{1}{192}&-\frac{1}{64}&\frac{5}{192}&\frac{17}{2880}\\[10pt]
({3}^{2})&\frac{1}{108}&\frac{1}{90}&\frac{1}{144}&\frac{1}{144}&-\frac{1}{324}&-\frac{1}{108}&-\frac{1}{324}&\frac{1}{864}&-\frac{1}{96}&-\frac{7}{864}&-\frac{19}{12960}\\[10pt]
({3}^{},{2}^{},{1}^{})&\frac{1}{36}&0&0&-\frac{1}{24}&-\frac{1}{108}&-\frac{1}{36}&-\frac{1}{108}&-\frac{1}{144}&0&\frac{7}{144}&\frac{1}{54}\\[10pt]
({3}^{},{1}^{3})&\frac{1}{108}&-\frac{1}{45}&-\frac{1}{72}&\frac{1}{36}&-\frac{1}{324}&-\frac{1}{108}&-\frac{1}{324}&-\frac{1}{108}&\frac{1}{16}&-\frac{1}{54}&-\frac{131}{6480}\\[10pt]
({2}^{3})&\frac{1}{288}&\frac{1}{240}&-\frac{1}{192}&-\frac{1}{192}&\frac{1}{864}&-\frac{1}{144}&-\frac{1}{108}&\frac{1}{576}&\frac{1}{192}&\frac{1}{144}&\frac{17}{4320}\\[10pt]
({2}^{2},{1}^{2})&\frac{1}{96}&-\frac{1}{80}&-\frac{1}{64}&-\frac{1}{64}&-\frac{1}{96}&0&\frac{1}{16}&\frac{1}{192}&\frac{1}{64}&0&-\frac{19}{480}\\[10pt]
({2}^{},{1}^{4})&\frac{1}{288}&-\frac{1}{80}&-\frac{1}{64}&\frac{5}{192}&-\frac{7}{864}&\frac{7}{144}&-\frac{1}{54}&\frac{1}{144}&0&-\frac{49}{576}&\frac{473}{8640}\\[10pt]
({1}^{6})&\frac{1}{4320}&-\frac{1}{720}&-\frac{7}{2880}&\frac{17}{2880}&-\frac{19}{12960}&\frac{1}{54}&-\frac{131}{6480}&\frac{17}{4320}&-\frac{19}{480}&\frac{473}{8640}&-\frac{473}{25920}
\end{array}\]
\caption{The values of $f_{\kappa\lambda}$, for partitions of  size 6.\label{table.fkl6}}
\end{table}

\note{For $|\kappa|=|\lambda|\le3$, the only nonzero terms $f_{\kappa\lambda}$ are
\begin{multline}
f_{(1)(1)} = 1; f_{(1,1)(1,1)} = -\frac14, f_{(2)(1,1)} = f_{(1,1)(2)} = \frac14,  f_{(2)(2)}=\frac14;\\
f_{(1,1,1)(1,1,1)} = \frac19, f_{(1,1,1)(2,1)}=f_{(2,1)(1,1,1)}=-\frac16, f_{(3)(1,1,1)}=f_{(1,1,1)(3)} = \frac{1}{18},\\
 f_{(3)(2,1)}=f_{(2,1)(3)}=\frac16,f_{(3)(3)}=\frac19.
\end{multline}
\note{Include table}
}
\subsection{Precomputation of  the $\VVV{r}{\mu\nu}$}
\label{sec.vrmn}
The $\VVV{r}{\mu\nu}$ were defined in Equation~\eqref{eqn.Vrmunu}. There is nothing to say actually, except that their computation depends on the previously computed $f_{\kappa \lambda}$, and simply consists of a finite sum. In addition, we have the relation $\VVV{r}{\mu\nu} = \VVV{r}{\nu\mu}$, which follows from the same symmetry for the $f_{\mu\nu}$.
\note{Include tables for $r$ vs. total weight in $\mu,\nu$}

\subsection{Precomputation of the $W_{\mu\nu}$} We simply truncate the Equation~\eqref{eqn.Wmunu}, page~\pageref{eqn.Wmunu}, which defined the $W_{\mu\nu}$. As this involves the $\coeff{n}{r}$ for larger and larger $n$, this converges exponentially fast. We still have the symmetry $W_{\mu\nu} = W_{\nu\mu}$.

\subsection{Alternative summation for $W_{\mu\nu}$} Careful consideration of Equation~\eqref{eqn.Wmunu} shows that  $W_{\mu\nu}$ is obtained as a sum with coefficients of a fixed derivative of the prime zeta function, evaluated at successive integers. Sums of this form can be efficiently reorganized, leading to new series with much better convergence. This is explained in  \cite{Cohen}, close in spirit to parts of  \cite{CFKRS2}.

\begin{comment}
% What follows has some mistakes
Let 
\begin{equation}
f(p,s) = \sum_m \frac{u(m)}{p^{m+s}},
\end{equation}
with 
\begin{equation}
S_f(s) = \sum_{p \text{ prime }} f(p,s).
\end{equation}
Then,
\begin{align}
S_f(s) &= \sum_{ m \ge 1} u(m) \primezeta(m+s)\\
&= \sum_{m \ge 1} u(m) \sum_{k \ge 1} \frac{\mu(k)}{k} \log \left(\zeta\left(k\left(m+s\right)\right)\right)\\
&=\sum_{N \ge 1} \left(\sum_{k | N} u\left(\frac{N}{k}\right) \frac{\mu(k)}{k}\right) \log \left(\zeta\left(N+ks\right)\right),
\end{align}
which follows from the change of variables $N = km$. 

Define, with $\psi_i:n \rightarrow n^i$ and $\cdot$ usual multiplication of functions, 
\begin{equation}
u^*_i(N) := \sum_{k | N} k^i \mu(k) \frac{N}{k} u\left(\frac{N}{k}\right)= ((\psi_i \cdot \mu)\star (\psi_1\cdot u))(N),
\end{equation}
for $\star$ the usual convolution of arithmetic functions. Then,  
\begin{equation}
S_f(0) = \sum_{N \ge 1} \frac{u_0^*(N)}{N} \log \left(\zeta\left(N\right)\right), 
\end{equation}
(with no problem if $u_0^*(1)=0$) and similarly $S_f^{(r)}(s)$ can be obtained from summing the series (with exponential decay) of the $u_i^*$ and derivatives of $\log(\zeta(N+ks))$ (obtained by the Fa\`a di Bruno formula for derivatives of $f(g(x))$). 
The values of $W_{\mu\nu}$ then follow from taking linear combinations of the $S_f^{(r)}(0)$.
\end{comment}
\subsection{Precomputation of the $d_{\mu\nu}$}
These were defined in Equation~\eqref{eqn.expansion}, page~\pageref{eqn.expansion}. A similar strategy as for the $f_{\kappa\lambda}$ works, except that we cannot do the computations exactly anymore: the coefficients depend on the $W_{\mu\nu}$, which can only be approximated. We thus need to change the base field of $\Lambda$ (or more accurately $\Lambda \otimes \Lambda$) from $\mathbb{Q}$ to a real field with large precision. Again, the symmetry $d_{\mu\nu} = d_{\nu\mu}$ is preserved.

\begin{appendix}
\label{appendix}
\section{Main Formulas and Definitions}
We now list some of the main formulas and definitions, together with an indication of the location where it appeared. It is hoped that this will help the reader understand the structure of the final result. 

The symbol ``$=$'' indicates a result. 
The symbol ``$:=$'' indicates an implicit definition, and the ``direction'' of this definition.
The symbol ``$\overset{?}{:=}$'' indicates an implicit definition, provided Conjecture~\ref{conj.pol} is true.
\begin{description}
\item[{Prime zeta function}]
\begin{align}
\primezeta (s) &= \sum_{n\ge 1} \frac{ \mu(n)}{n}  \log \zeta(ns) \quad \Re s > 1\refer{eqn.moebius}\\
\primezeta(r+s) &=:  \sum_{n\ge 0} \coeff{r}{n} s^n \quad r \in \mathbb{N}\setminus \{0,1\}  \refer{eqn.coeffrn}\\
\primezeta(1+s) &=:  -\log(s) + \sum_{n\ge 0} \coeff{1}{n} s^n  \refer{eqn.coeff1n}
\end{align}

\item[Transition between bases/changes of alphabets/partition invariants]
\begin{align}
z_\lambda &:= \prod_{j\ge 1}  \left(j^{M_j(\lambda)}M_j(\lambda)!\right) \refer{eqn.zlambda}\\
\symssymb{\lambda} &= \sum_{\mu \vdash |\lambda|} \frac{1}{z_\mu} \chi^{\lambda}(\mu) \sympsymb{\mu}\refer{eqn.transitionschur}\\
\sympsymb{\mu}&=\sum_{\lambda\vdash |\mu|} \chi^{\lambda}(\mu) \symssymb{\lambda}\refer{eqn.transitionpower}\\
\symp{m}{\phi_\kappa(\AAA)} & =  \sum_{\mu\vdash m } \binom{m}{\mu} \symm{\mu}{\kappa}  \symp{\mu}{\AAA} |\AAA|^{l(\kappa)-l(\mu)} \refer{eqn.pmphik}
\end{align}

\item[Final result]
\begin{align}
\int_0^T \left|\zeta\left(\frac{1}{2} + \mathfrak{i}t\right)\right|^{2k} \ud t &\overset{?}{=:} \int_0^T P_k\left(\log \frac{t}{2\pi}\right) \ud t + O\left(T^{1/2+\epsilon}\right)\refer{eqn.momentpol}\\
\refer{eqn.crk}
P_k(x) &=: c_0(k) x^{k^2} + c_1(k) x^{k^2-1} +\cdots + c_{k^2}(k)\\
c_N(k) &= \frac{1}{(k^2-N)!} \sum_{\substack{\kappa,\lambda \\ |\kappa|+|\lambda|=N}}  d_{\kappa\lambda} \dim(\lambda, \complementary{k}{\kappa})
\refer{eqn.main}
\end{align}

\item[Dimension]
\begin{align}
\refer{eqn.dim}
\dim(\lambda, \complementary{k}{\kappa})& = \frac{g_k B(k)}{k^2\cdot (k^2-1)\cdots(k^2-|\kappa|-|\lambda|+1)}, \quad B(k) \in \mathbb{Q}(k)
\end{align}

\item[(Expected-to-be-)transcendental coefficients]
\begin{align}
  \sum_{\kappa,\lambda} d_{\kappa\lambda} \syms{\kappa}{\XXX} \syms{\lambda}{\YYY} &:=  \exp \left( \sum_{\mu,\nu} W_{\mu\nu} \symp{\mu}{\XXX}\symp{\nu}{\YYY}\right)
  \refer{eqn.expansion}
  \end{align}
 \begin{align}
W_{\mu\nu} &:= \sum_{r \ge 1} \VVV{r}{\mu\nu}   \coeff{r}{|\mu|+|\nu|}
\refer{eqn.Wmunu}
\\
\VVV{r}{\mu\nu} & :=\binom{|\mu|+|\nu|}{\mu\cup \nu} \sum_{\kappa,\lambda\vdash r}  f_{\kappa \lambda}   \symm{\mu}{\kappa}  \symm{\nu}{\lambda} k^{l(\kappa)+l(\lambda)-l(\mu)-l(\nu)}
\refer{eqn.Vrmunu}
\end{align}
\begin{align}
\refer{eqn.fkl}
 \sum_{\substack{\kappa,\lambda\\|\kappa| = |\lambda|}} \frac{1}{z_\kappa z_\lambda} \symp{\kappa}{\XXX}  \symp{\lambda}{\YYY}  T^{|\kappa|} &=:  \exp \left( \sum_{\substack{\kappa,\lambda\\|\kappa| = |\lambda| \ge 1}} f_{\kappa \lambda} \symp{\kappa}{\XXX}  \symp{\lambda}{\YYY}  T^{|\kappa|} \right)
\end{align}
\end{description}
\end{appendix}

\begin{comment}
% One variable version, false
\begin{appendix}
\begin{align}
  \sum_{\lambda} d_{\lambda}\syms{\lambda}{\XXX} &:=  \exp \left( \sum_{\mu} W_{\mu} \symp{\mu}{\XXX}\right)\\
\sum_{n \ge 0} \left( \sum_{\lambda \vdash n} \frac{1}{ z_\lambda}  \symp{\lambda}{\XXX}  T^{n}\right)^2 &=:  \exp \left( \sum_{\substack{\lambda\\ |\lambda| \ge 2}} f_{\lambda}   \symp{\lambda}{\XXX}  T^{|\lambda|} \right)\\
\VVV{r}{\mu} & :=\binom{|\mu|}{\mu} \sum_{\lambda\vdash 2r}  f_{\lambda}   \symm{\mu}{\lambda} k^{l(\lambda)-l(\mu)}\\
W_{\mu} &:= \sum_{r \ge 1} \VVV{r}{\mu}   \coeff{r}{|\mu|}
\end{align}
\end{appendix}
\end{comment}

\bibliographystyle{halpha} 
\bibliography{references}

\end{document}